\documentclass{amsart}
\newcommand{\univ}{\mathcal{U}}
\newcommand{\hprop}{\mathsf{hProp}}
\newcommand{\Omegann}{\Omega_{\neg\neg}}
\newcommand{\markov}{\mathbf{MP}}
\newcommand{\markovind}{\mathbf{MI}}
\newcommand{\issuc}{\mathtt{IsSuc}}
\newcommand{\ect}{\mathbf{ECT}}
\newcommand{\asm}{\mathbf{Asm}}
\newcommand{\sets}{\mathbf{Set}}
\newcommand{\op}{\mathrm{op}}
\newcommand{\cubasm}{\mathbf{Asm}^{\square^\op}}
\newcommand{\ZZ}{\mathbb{Z}}
\newcommand{\fibre}{\mathtt{hFibre}}
\newcommand{\dom}{\operatorname{dom}}
\newcommand{\dec}{\mathtt{Dec}}
\newcommand{\erase}{\mathtt{erase}}
\newcommand{\sone}{\mathbb{S}^1}

\usepackage{amsmath,amsthm,amssymb}
\usepackage{hyperref}
\usepackage{todonotes}
\usepackage[all]{xy}
\usepackage{tikz-cd}

\newtheorem{thm}{Theorem}[section]
\newtheorem{prop}[thm]{Proposition}
\newtheorem{lemma}[thm]{Lemma}
\newtheorem{cor}[thm]{Corollary}
\newtheorem{rmk}[thm]{Remark}

\newtheorem{axiom}[thm]{Axiom}

\theoremstyle{definition}
\newtheorem{example}[thm]{Example}
\newtheorem{defn}[thm]{Definition}
\newtheorem{notation}[thm]{Notation}

\newcommand{\inl}{\mathtt{inl}}
\newcommand{\inr}{\mathtt{inr}}

\newcommand{\NN}{\mathbb{N}}
\usetikzlibrary{snakes}

\begin{document}
\title{Oracle Modalities}
\author{Andrew W Swan}
\thanks{This material is based upon work supported by the Air Force Office of Scientific Research under award number FA9550-21-1-0024.}
\date{\today}
\begin{abstract}
  We give a new formulation of Turing reducibility in terms of higher modalities, inspired by an embedding of the Turing degrees in the lattice of subtoposes of the effective topos discovered by Hyland. In this definition, higher modalities play a similar role to I/O monads or dialogue trees in allowing a function to receive input from an external oracle. However, in homotopy type theory they have better logical properties than monads: they are compatible with higher types, and each modality corresponds to a reflective subuniverse that under suitable conditions is itself a model of homotopy type theory.

  We give synthetic proofs of some basic results about Turing reducibility in cubical type theory making use of two axioms of \emph{Markov induction} and \emph{computable choice}. Both axioms are variants of axioms already studied in the effective topos. We show they hold in certain reflective subuniverses of cubical assemblies, demonstrate their use in some simple proofs in synthetic computability theory using modalities, and show they are downwards absolute for oracle modalities. These results have been formalised using cubical mode of the Agda proof assistant.

  We explore some first connections between Turing reducibility and homotopy theory. This includes a synthetic proof that two Turing degrees are equal as soon as they induce isomorphic permutation groups on the natural numbers, making essential use of both Markov induction and the formulation of groups in HoTT as pointed, connected, 1-truncated types. We also give some simple non-topological examples of modalities in cubical assemblies based on these ideas, to illustrate what we expect higher dimensional analogues of the Turing degrees to look like.
\end{abstract}
\maketitle

\section{Introduction}
\label{sec:introduction}

\subsection{Synthetic computability theory}
\label{sec:synth-comp-theory}

Synthetic computability \cite{richmanwithouttears, bauerfirststeps, bauerfixedpoint, forsterthesis} is an approach to computability where instead of working directly with explicit descriptions of objects, one works with simpler, more natural constructions in constructive mathematics, which can then be interpreted in realizability models, to recover the original versions. Since the Turing degrees are a major topic within computability theory, it is natural to ask what is the best way to talk about Turing reducibility synthetically. Recently, multiple synthetic definitions of Turing reducibility have been developed \cite{bauerexcursion, forsterkirstmuck}. We will consider a new variation of one of the earliest synthetic approaches to Turing reducibility: Bauer's suggestion in the conclusion to \cite{bauerfirststeps} of applying Hyland's embedding of the Turing degrees in the lattice of local operators in the effective topos \cite[Section 17]{hylandeff}.

\subsection{Oracles and monads}
\label{sec:oracles-monads}

Before introducing subtoposes and modalities, we first consider some related constructions using monads, for comparison.

One of the classic examples of applications of monads in programming due to Moggi in \cite[Example 1.1]{moggi91} is that of \emph{interactive input}. The idea is to model programs that can call a function to accept input of type \(B\). The input function could be called at any point, multiple times or not at all. The appropriate monad in this case is the free monad \(T\) on the polynomial endofunctor \(X \mapsto X^B\).\footnote{See e.g. \cite[Section 6]{gambinohylanddepw} for details on free monads on polynomial endofunctors.} Concretely we can describe \(T X\) as the type of \(B\)-branching trees with leaves from \(X\), or as the inductive type generated by constructors \(\eta_X : X \to T X\) and \(i : (B \to T X) \to T X\). We represent programs that do not request input as \(\eta_X(x)\). We think of an element of the form \(i(f)\) where \(f : B \to T X\) as a program that first accepts an element \(b\) of \(B\) as input, and then continues using \(f\), returning \(f(b)\) as output.

This idea of modelling input and output through free monads on polynomial endofunctors was further developed by Hancock and Setzer in \cite{hancocksetzer}. In particular, instead of just one input function, we can have a collection of input functions indexed by a type \(A\), making \(B\) a family of types \(B : A \to \univ\). Note that \((A, B)\) is just a (non dependent) polynomial and the inductive type in \cite[Section 2]{hancocksetzer} is the free monad on the corresponding polynomial endofunctor. This idea can also be understood through Escard\'{o}'s notion of dialogue tree \cite{escardoeffectfulforcing}, which as he points out is equivalent to an earlier notion due to Kleene \cite{kleenerfqft1}. The same idea is also related to the synthetic definition of Turing reducibility due to Forster, Kirst and M\"{u}ck, as explained in \cite[Section 2]{forsterkirstmuck}.

We can view Turing machines with oracle as a special case of this construction. Suppose we are given a function \(\chi : \NN \to \NN\) and for each natural number \(n\) we have a query operation, where the Turing machine can output a request for the value of \(\chi\) at \(n\) and then receive as input the value \(\chi(n)\). We can describe the polynomial as follows: The type of constructors is \(\NN\) and an element of \(B(n)\) is a pair consisting of a number \(m\) and a proof that \(\chi(n)\) is defined and equal to \(m\).

\subsection{Oracles, Lawvere-Tierney topologies, and subtoposes}
\label{sec:oracles-subtoposes}

One the main motivations for this paper is the classic result due to Hyland \cite[Section 17]{hylandeff}, and further developed by Phoa \cite{phoarelcompeff} and Van Oosten \cite{vanoostenrelrec}, that the Turing degrees embed into the lattice of Lawvere-Tierney topologies of the effective topos. By the general theory of toposes \cite[Sections A4.3-4.4]{theelephant} each Lawvere-Tierney topology in a topos \(\mathcal{E}\) gives rise uniquely to a subtopos of \(\mathcal{F} \hookrightarrow \mathcal{E}\) together with a sheafification operator that reflects \(\mathcal{E}\) down to \(\mathcal{F}\). In fact sheafification is an idempotent monad with \(\mathcal{F}\) the category of algebras. The sheafification monads for the subtoposes considered by Hyland are somewhat related to the monads considered in Section~\ref{sec:oracles-monads}. However, they behave differently in two important respects.

\subsubsection*{Subtoposes are toposes} Each subtopos of the effective topos is a topos in itself. As such, any proof in higher order logic can carried out not just in the effective topos itself, but also in all of its subtoposes. In this paper, we will refer to this idea as \emph{relativisation}, viewing it as closely related to relativisation in computability theory. In contrast, the categories of algebras for the monads in Section~\ref{sec:oracles-monads} do not need to be toposes, or even cartesian closed.

\subsubsection*{Subtoposes are subcategories} The definition of sheaf for a Lawvere-Tierney topology \(j\) is that it has a \emph{unique} lift against all \(j\)-dense maps. As such, we can think of sheaves as objects satisfying a \emph{property}, rather than objects with additional \emph{structure}. Hence the natural notion of morphism between sheaves is just any map between the objects. In contrast, the algebras from Section~\ref{sec:oracles-monads} have additional structure which is not preserved by all maps.

In this way, we can think of sheafification as having better logical and categorical properties than the monads of Section~\ref{sec:oracles-monads}, while retaining the good computational properties.

\subsection{Homotopy type theory and higher modalities}
\label{sec:modal-homot-type}

In order to generalise Hyland's result from topos theory/extensional type theory to higher topos theory/homotopy type theory, the natural approach is that of \emph{higher modalities}. Higher modalities have a simple definition in homotopy type theory, leading to a rich and elegant theory as developed by Rijke, Shulman and Spitters \cite{rijkeshulmanspitters}. This general theory applies without issue to arbitrary types without e.g. needing to restrict anywhere to \(0\)-truncated types. This is in contrast to the theory of monads, which has proved difficult to formalise in homotopy type theory due to coherence issues.\footnote{For instance, such issues are discussed by Finster et al. in \cite{finsteralliouxsozeau} where they propose a somewhat elaborate definition of monad requiring an extension of type theory.}

Rijke et al. showed in \cite[Section 3.3]{rijkeshulmanspitters} that Lawvere-Tierney topologies naturally generalise to a kind of modality that they refer to as \emph{topological modality}. Topological modalities are a special case of \emph{nullification}, which they showed can be constructed as a higher inductive type. Nullification takes as input a type \(A : \univ\) together with a family of types \(B : A \to \univ\), i.e. a polynomial. If we examine how the nullification \(\bigcirc X\) of a type \(X\) is defined (in \cite[Section 2.2]{rijkeshulmanspitters}) we see that the point constructors are the same as the constructors of the free monad on the polynomial endofunctor \(P_{(A, B)}\) that we saw in Section~\ref{sec:oracles-monads}. However, it also has path constructors, which play an important role in the general theory. Just as a Lawvere-Tierney topology gives rise to a subtopos of a topos, in HoTT a topological modality gives us a reflective subuniverse, which is itself a model of HoTT.

In particular given \(a : A\) and a map \(f : B(a) \to \NN\), we freely add a new natural number \(\mathtt{sup}(a, f)\) to \(\bigcirc \NN\). The path constructors allow us, in particular, to construct a path \(\mathtt{sup}(a, f) = f(b)\) given an element \(b\) of \(B(a)\), and thereby a proof of \(\sum_{n : \NN} \mathtt{sup}(a, f) = \eta(n)\). On the other hand, if we are only given an element of \(\neg\neg B(a)\), then we know using classical logic that there is \(n : \NN\) such that \(\mathtt{sup}(a, f) = n\), but are not able to find \(n\) constructively. Although we are not necessarily able to find the witness \(n\) in general, the path constructor still has a noticeable effect. Any well defined map out of \(\bigcirc \NN\) has to respect the path constructors. In this way \(\mathtt{sup}(a, f)\) is ``indistinguishable'' from elements of \(\bigcirc \NN\) of the form \(\eta(n)\). For example, we can always show that a map \(\bigcirc \NN \to \bigcirc 2\) is constant as soon as its composition with \(\eta\) is constant.

\subsection*{Acknowledgements}
I'm grateful for useful comments, discussions and suggestions by Mathieu Anel, Carlo Anguili and Egbert Rijke.

\section{Methodology, terminology and notation}
\label{sec:background}

Before giving our results, we will review the existing work that we will build on, and fix some terminology and notation.

\subsection{Homotopy type theory and cubical type theory}
\label{sec:homotopy-type-theory}

We will make use of homotopy type theory (HoTT), a relatively new approach to type theory that draws on ideas from homotopy theory. The key new idea in HoTT is that identity types can have interesting, highly non trivial structure. For example, given two proofs of equality \(p, q : x = y\), it was common traditionally in type theory to ignore the possibility that \(p \neq q\) or even to assume that any two equality proofs must be equal (sometimes referred to as \emph{axiom \(K\)}). However, in HoTT we visualise equality proofs as paths in a topological space up to homotopy. There are many examples of spaces containing paths that are not homotopic, the simplest example being the circle, which has for each point \(x\) and each \(n : \ZZ\) a path from \(x\) to itself that loops round the circle \(n\) times.

Formally, we obtain this structure through \emph{higher inductive types} (HITs) and the axiom of \emph{univalence}. Inductive types are used throughout type theory, and can be informally described as the simplest or ``least'' types that are closed under certain rules telling us how to construct new elements from old ones. HITs generalise this idea by allowing us to not only construct new elements, but also new proofs of equality between elements. We refer to the former as \emph{point constructors} and the latter as \emph{path constructors}. Many spaces studied in homotopy theory can be constructed as HITs. For example, the circle is the HIT generated by a single point \(x_0\) and a path \(x_0 = x_0\). The univalence axiom tells us that the identity type on the universe of small types, which we will write as \(\univ\), is given by equivalence of types. That is, for \(A, B : \univ\) the type \(A = B\) is equivalent to the type of equivalences from \(A\) to \(B\). In particular univalence tells us that equality can have computational meaning: e.g. a proof of \(\NN = \NN\) is precisely a permutation of \(\NN\).

We will assume the reader is familiar with standard concepts from homotopy type theory, such as path induction, transport, hlevels, truncation, the homotopy fibres of a map, equivalences and univalence which can all be found in standard references such as \cite{hottbook} and \cite{rijkebook}.
\begin{notation}
  \begin{enumerate}
  \item We will refer to \(-1\)-truncated types as \emph{propositions} and \(0\)-truncated types as \emph{sets}.
  \item For a family of types \(B : A \to \univ\) and a path \(p : a = a'\) in \(A\), we will write \(p^\ast\) for the map \(B(a') \to B(a)\) given by transport.
  \item For \(f : A \to B\) and \(b : B\) we write the homotopy fibre of \(b\) as \(\mathtt{hFibre}_f(x)\).
  \item We write \(\mathtt{Dec}(A)\) to mean \(A + \neg A\).
  \item For a type \(X\) with element \(x : X\) we write the loop space \(x = x\) as \(\Omega(X, x)\).
  \end{enumerate}
\end{notation}

\emph{Cubical type theory} \cite{coquandcubicaltt, abchfl} is a later development which adds computational rules to type theory that can be used to construct HITs and prove univalence, while retaining good computational properties of type theory such as canonicity\cite{hubercanonicity} and normalisation\cite{sterlingangiuli}. We have formalised the main results of this paper\footnote{The formalisation is available at \url{https://github.com/awswan/oraclemodality}.} in cubical Agda \cite{cubicalagda}, a proof assistant based on cubical type theory.

\subsection{Modalities}
\label{sec:modalities}

We will make heavy use of theory of higher modalities, as developed in HoTT by Rijke, Shulman and Spitters. We recall some definitions, notation and statements of theorems below, but refer the reader to \cite{rijkeshulmanspitters} for full details.

\begin{defn}
  A \emph{modal operator} consists of a map \(\bigcirc : \univ \to \univ\) together with a family of maps \(\eta_X : X \to \bigcirc X\) for \(X : \univ\). We call \(\eta\) the \emph{unit} of the modal operator.

  We say a type \(X\) is \emph{\(\bigcirc\)-modal} if \(\eta_X : X \to \bigcirc X\) is an equivalence. We say a map is \emph{\(\bigcirc\)-modal} if its fibres are \(\bigcirc\)-modal. We write the type of all small \emph{\(\bigcirc\)-modal} types as \(\univ_\bigcirc\). When \(X\) is a proposition we will also refer to \(\bigcirc\)-modal as \emph{\(\bigcirc\)-stable}.

  We say \(X\) is \emph{\(\bigcirc\)-connected} if \(\bigcirc X\) is contractible. We say a map is \emph{\(\bigcirc\)-connected} if its fibres are \(\bigcirc\)-connected. When \(X\) is a proposition we will also refer to \(\bigcirc\)-connected as \emph{\(\bigcirc\)-dense}.

  We say \(X\) is \emph{\(\bigcirc\)-separated} if for all \(w, z : \bigcirc X\), the identity type \(w = z\) is \(\bigcirc\)-modal.
\end{defn}

\begin{defn}
  A \emph{higher modality} consists of a modal operator \(\bigcirc, \eta\) together with terms \[\mathtt{elim}_X : \prod_{x : X} \bigcirc(P(\eta(x))) \to \prod_{z : \bigcirc X}\bigcirc(P(z))\] for each \(X : \univ\) and \(P : \bigcirc X \to \univ\) and \[\mathtt{elim\beta}_X(f, x) : \mathtt{elim}_X(f)(\eta(x)) = f(x)\] for each \(f : \prod_{x : X} \bigcirc(P(\eta(x)))\) and \(x : X\), and such that \(x = y\) is \(\bigcirc\)-modal for \(x, y : \bigcirc X\).

  We will usually refer to higher modalities just as \emph{modalities}.
\end{defn}

\begin{thm}
  If \(\bigcirc\) is a modality, then \(\univ_\bigcirc\) is closed under \(\Pi\)-types, \(\Sigma\)-types and identity types. We say that the \(\bigcirc\)-modal types form a \emph{\(\Sigma\)-closed reflective subuniverse} of \(\univ\).
\end{thm}

\begin{defn}
  A modality \(\bigcirc\) is \emph{lex} if \(\univ_\bigcirc\) is \(\bigcirc\)-modal.
\end{defn}

\begin{defn}
  Suppose we are given a polynomial, i.e. a type \(A : \univ\) together with a family of types \(B : A \to \univ\). A type \(X\) is \(B\)-\emph{null} if for all \(a : A\), the constant function map \(X \to X^{B(a)}\) is an equivalence.
\end{defn}

\begin{thm}
  Given a polynomial \((A, B)\), there is a modality \(\bigcirc\) such that a type \(X\) is \(\bigcirc\)-modal precisely if it is \(B\)-null. We refer to \(\bigcirc X\) as the \emph{nullification} of \(X\).
\end{thm}

\begin{thm}
  For families of propositions \(a : A \vdash B(a) : \hprop\), nullification is a lex modality. In this case we will also refer to nullification as \emph{sheafification}. Nullification restricts to a Lawvere-Tierney topology \(j\) on propositions, and agrees with the usual notion of \(j\)-sheafification on sets. We refer to such modalities as \emph{topological} modalities.
\end{thm}

For topological modalities, reflective subuniverses can be viewed as models of homotopy type theory, as worked out in detail by Quirin in \cite[Section 4.5]{quirinthesis}.

\subsection{Partial functions in type theory}

We use an approach to partial functions in type theory mostly similar to well known formulations such as \cite{rosolinithesis, escardoknapp}, with some minor adjustments.

We first consider the most general possible definition of partial element of a type, which is nonetheless sufficient to fix some notation that we will use throughout the paper.
\begin{defn}
  \label{def:partialelt}
  A \emph{partial element} of a type \(X\) is a type \(D : \univ\) together with a map \(D \to X\).
\end{defn}

\begin{notation}
  Given a partial element \(\xi := (D, f)\) of \(X\), we write \(\xi {\downarrow}\) for \(D\), and \(\xi \downarrow= x\) to mean \(\sum_{d : D} f(d) = x\). We will refer to \(D\) as the \emph{domain} of \(\xi\), and say that \(\xi\) is \emph{defined} when \(D\) is inhabited. If we have a witness, \(d\) of \(\xi {\downarrow}\) that is clear from the context, we will sometimes refer to \(f(d)\) just as \(\xi\).
\end{notation}

Note that the type of all partial elements of \(X\) does not belong to the universe \(\univ\) of small types, and in general \(D\) does not need to be a proposition. Because of this, it is usually more convenient to restrict to the subtype of partial elements where the domain is a \(\neg \neg\)-stable proposition, formally defined as \(\partial X\) below. We will restrict to the case when class of all \(\neg \neg\)-stable propositions is a small type \(\Omegann\) (see Section \ref{sec:cubical-assemblies-1} below), making \(\partial X\) a small type.
\begin{defn}
  For a type, \(X\), we will write \(\partial X\) for \(\sum_{P : \Omega_{\neg \neg}} X^P\).
\end{defn}
We can think of \(\partial X\) as the collection of well defined partial elements with \(\neg\neg\)-stable domain.

\subsection{Assemblies}
\label{sec:assemblies}

We will use realizability to justify the non standard axioms that we assume, and to relate our ideas to computability, following earlier work in synthetic computability theory, particularly that of Bauer \cite{bauerfirststeps}. The realizability models we will consider are all based on the category of \emph{assemblies}. We give some basic definitions below, but the reader is referred to the standard reference of \cite{vanoosten} for more details about on the categorical approach to realizability that we use here. For this paper we will only consider realizability over the first Kleene algebra \(\mathcal{K}_1\), i.e. Turing computable functions.

\begin{notation}
  We write \(\varphi_e\) for the partial computable function coded by the natural number \(e\). For proofs about the realizability model will sometimes write \(\varphi_e(n)\) as \(e n\) following \cite{vanoosten}.
\end{notation}

\begin{defn}
  An \emph{assembly} is a pair \(A = (| A |, \Vdash)\) where \(| A |\) is a set and \(\Vdash\) is a binary relation \(\Vdash \,\subseteq\, \NN \times | A |\) such that for all \(a : | A |\) there exists \(e : \NN\) such that \(e \Vdash a\). We say that \(e\) \emph{realizes} \(a\).

  Assemblies form a category, which we will denote \(\asm\) with morphisms \((| A |, \Vdash_A) \to (| B |, \Vdash_B)\) defined as a function \(f : | A | \to | B |\) such that there exists \(e : \NN\) such that whenever \(a : | A |\) and \(d \Vdash a\), we have \(e d {\downarrow}\) with \(e d \Vdash f(a)\). We say that \(e\) \emph{tracks} or \emph{realizes} \(f\).
\end{defn}

The category of assemblies is locally cartesian closed, and thereby a model of extensional type theory.

\begin{defn}
  We say an assembly \((A, \Vdash)\) is \emph{uniform} if there is \(e : \NN\) such that \(e \Vdash a\) for all \(a : A\).
\end{defn}

\begin{rmk}
  Every uniform assembly is isomorphic to one where the set of realizers of each \(a : A\) is the singleton \(\{0\}\). For convenience, we will always assume this is the case when working with a uniform assembly.
\end{rmk}

Note that we can view any set as a uniform assembly, and that this defines a functor \(U : \sets \to \asm\).\footnote{In many texts, including \cite{vanoosten}, \(U\) is written \(\nabla\). However, in this paper we will reserve \(\nabla\) for the corresponding modality \(\Delta U\) on \(\asm\), and the analogous construction in cubical assemblies.}

\subsection{Cubical assemblies}
\label{sec:cubical-assemblies-1}

Cubical sets \cite{bchcubicalsets, bchunivalence, coquandcubicaltt, awodey19, abchfl} can be used to give models of HoTT within a constructive setting. In particular, we can define cubical sets inside the locally cartesian closed category of assemblies, to obtain \emph{cubical assemblies}, \(\cubasm\). The first complete model of HoTT in cubical assemblies was developed by Uemura \cite{uemuracubasm}, using the techniques developed by Licata, Orton, Pitts and Spitters \cite{pittsortoncubtopos, lops} where one first checks cubical assemblies are locally cartesian closed, and then carries out almost all of the interpretation of HoTT within the internal logic of cubical assemblies.

The author and Uemura further developed cubical assemblies in \cite{uemuraswan}. We showed that Church's thesis, the axiom that all functions \(\mathbb{N} \to \mathbb{N}\) are computable, is false in cubical assemblies, in contrast to many earlier, well known realizability models such as assemblies and the effective topos, where it holds. However, we also showed that Church's thesis is nonetheless consistent with HoTT, by constructing a reflective subuniverse of cubical assemblies where it does hold.

We will make use some more recent work on cubical assemblies by the author, who in \cite{swandnsprop} gave a construction of a classifier for \(\neg\neg\)-stable propositions in cubical assemblies. This justifies the following axiom, which we will assume throughout the paper.
\begin{axiom}
  \label{ax:dnpropresize}
  There is a small type \(\Omega_{\neg\neg} : \univ\) together with a map \(| \cdot | : \Omega_{\neg\neg} \to \univ\) such that 
  \begin{enumerate}
  \item \(|z|\) is a \(\neg\neg\)-stable proposition for all \(z : \Omega_{\neg \neg}\).
  \item Every \(\neg\neg\)-stable proposition is of the form \(| z |\) for a unique \(z : \Omega_{\neg \neg}\).
  \end{enumerate}
\end{axiom}

\begin{rmk}
  Readers familiar with approaches to synthetic computability theory based in CIC, such as \cite{forsterthesis}, may find it helpful to think of \(\Omega_{\neg\neg}\) as playing the same role that the impredicative universe of propositions, \(\mathbb{P}\) plays in CIC.
\end{rmk}

We will use the fact that \(\Omega_{\neg\neg}\) not only exists in cubical assemblies, but can be described explicitly as follows. First recall that assuming the law of excluded middle in our metatheory, we can describe the classifier for \(\neg\neg\)-stable propositions in assemblies as the uniform assembly on \(2\), \(U 2\). Write \(\Delta\) for the constant presheaves functor \(\asm \to \cubasm\).

\begin{thm}[{\cite[Example 4.5]{swandnsprop}}]
  \label{thm:deltpresdnclass}
  \(\Delta (U 2)\) is a classifier for \(\neg\neg\)-stable propositions in \(\cubasm\).
\end{thm}

Uemura showed in \cite{uemuracubasm} that the universe constructed therein does not satisfy propositional resizing. Moreover the author will show in a future paper that in fact there is no subobject classifier in the basic\footnote{It remains possible that there is a reflective subuniverse of cubical assemblies with a subobject classifier, although as yet none has been constructed.} cubical assemblies model. This suggests it is best to avoid assuming it for best compatibility with realizability models of HoTT.

Our first use of \(\Omega_{\neg\neg}\) is in allowing us to formulate \emph{Extended Church's thesis}. This is essentially the same formulation as given in \cite[Definition 6.1]{swandnsprop}, except that following the standard convention in synthetic computability theory, we avoid explicit reference to primitive recursive functions by postulating a function \(\varphi : \NN \times \NN \times \NN \to \NN + \{\bot\}\) where we write the value of the function as \(\varphi_e^k(n)\) for \(e, k, n \in \NN\), and then define \(\varphi_e : \NN \to \partial \NN\) by setting \(\varphi_e(n) \downarrow= m\) to be \(\sum_{k : \NN} (\varphi_e^k(n) = m)\).
\begin{axiom}
  Extended Church's thesis, \(\ect\) is the axiom asserting a function \(\varphi_e^k(n) : \NN + \bot\) for \(e, k, n \in \NN\) such that for \(e, n : \NN\) there is at most one \(k\) such that \(\varphi_e^k(n) \in \NN\), and
  \[
    \prod_{f : \NN \to \partial \NN} \left \| \sum_{e : \NN} \prod_{n : \NN} f(n) {\downarrow} \to \varphi_e(n) \downarrow= f(n) \right \|.
  \]
\end{axiom}

\begin{thm}[{\cite[Theorem 6.4]{swandnsprop}}]
  There is a reflective subuniverse of cubical assemblies where \(\ect\) holds. Hence \(\ect\) is consistent with HoTT.
\end{thm}

\section{Oracle modalities in homotopy type theory}
\label{sec:oracle-modal-cubic}

\subsection{\(\neg\neg\)-Sheafification}
\label{sec:negn-sheaf-sets}

In order to do synthetic computability theory, we need to have some way to talk about non-computable functions that is compatible with Church's thesis. We will do this using \(\neg\neg\)-sheafification, motivated by the fact that \(\sets\) can be seen as the subtopos of \(\neg\neg\)-sheaves in the effective topos. We can understand \(\neg\neg\)-sheafification, \(\nabla\), intuitively as follows. Given a type \(X\), propositional truncation \(\| X \|\) freely makes \(X\) into a proposition by adding paths between points, homotopies between paths, etc. However, truncation does not remove any computational information. Given a global section of \(\|X\|\) we can find computably a global section of \(X\). Double negation both truncates and erases computational information: \(\neg \neg X\) is a proposition, and is interpreted as a uniform object in cubical assemblies. Finally, \(\neg\neg\)-sheafification erases computational information, but does not truncate: \(\nabla X\) is interpreted as a uniform object, but is not necessarily \(-1\) truncated.

As pointed out in \cite{rijkeshulmanspitters}, \(\neg\neg\)-sheafification is a special case of nullification, assuming propositional resizing. However, as explained in Section \ref{sec:cubical-assemblies-1} we wish to avoid using propositional resizing in this paper in favour of the weaker \(\neg\neg\)-resizing. Hence we use an alternative construction, based on sheafification in topos theory \cite[Section V.3]{moerdijkmaclane}, which preserves universe level while only requiring a classifier for \(\neg\neg\)-stable propositions. This also has the advantage that we can easily describe explicitly how the definition unfolds in cubical assemblies.
The approach here does have the disadvantage that it is not the true \(\neg\neg\)-sheafification, but only its \(0\)-truncation, which we indicate by writing the modality as \(\nabla_{0}\). However, this is sufficient for the uses of \(\neg\neg\)-sheafification in this paper, where it will only be applied to sets, and often just to \(\neg\neg\)-separated sets.

\begin{defn}
  We say a proposition \(P\) is \emph{\(\neg\neg\)-dense} if \(\neg\neg P\).
  
  We say a type \(X\) is a \emph{\(\neg\neg\)-sheaf} if for every \(\neg\neg\)-dense proposition, \(P\), the constant function map \(X \to X^P\) is an equivalence.
\end{defn}

\begin{defn}
  \label{defn:dnsheafification}
  Given a type \(X\), we define \(\nabla_0 X\) to consist of \(\neg\neg\)-stable subsets of \(X\), \(\xi : X \to \Omegann\), with the following properties:
  \begin{enumerate}
  \item \(\xi\) is not empty: \(\neg \neg \sum_{x : X} \xi(x)\)
  \item \(\xi\) has at most one element up to double negation, \(\prod_{x, y} \xi(x) \to \xi(y) \to \neg \neg x = y\)
  \end{enumerate}
\end{defn}

\begin{rmk}
  We can also think of the elements of \(\nabla_0 X\) as partial elements of \(X\), as in Definition~\ref{def:partialelt}, by taking the domain of \(\xi : X \to \Omegann\) to be \(\sum_{x : X} \xi(x)\), with the map \(\sum_{x : X} \xi(x) \to X\) simply given by projection. In particular we can use the same notation for partial elements as usual. In general \(\xi {\downarrow}\) does not need to be a proposition. However, we will typically restrict to the case where \(X\) is a \(\neg\neg\)-separated set, and this precisely ensures that \(\xi {\downarrow}\) is in fact a proposition. We can roughly think of \(\nabla_{0} X\) as the type of partial elements whose domain is a \(\neg \neg\)-dense proposition. Compare this to \(\partial X\) whose elements have \(\neg\neg\)-stable domain.
\end{rmk}

\begin{thm}
  \(\nabla_0 : \univ \to \univ\) is a modality. It reflects onto \(\neg\neg\)-sheaves that are \(\neg\neg\)-separated.
\end{thm}

\begin{proof}
  First to show that each \(\nabla_0 X\) is \(\neg\neg\)-separated, observe that \(\nabla_0 X\) embeds into an exponential of \(\Omega_{\neg\neg}\), which is clearly \(\neg\neg\)-separated.
  
  We now verify that \(\nabla_0 X\) is a \(\neg\neg\)-sheaf for all \(X\). Suppose we are given \(f : P \to \nabla X\) where \(P\) is a \(\neg\neg\)-dense proposition. We can define an element of \(\Omega_{\neg\neg}^X\) by simply sending \(x : X\) to \(\prod_{z : P} f(z) \downarrow= x\) which is a product of \(\neg\neg\)-stable propositions, and so \(\neg\neg\)-stable. To check \(\neg\neg \sum_{x : X} \prod_{z : P} f(z) \downarrow= x\), observe that we can derive \(\sum_{x : X} \prod_{z : P} f(z) \downarrow= x\) from \(z : P\) together with a proof of \(\sum_{x : X} f(z) \downarrow= x\) and that both of these are \(\neg\neg\)-dense. To show there is at most one element up to double negation, similarly observe that we can derive that two elements are equal from \(\neg\neg\)-connected propositions. Finally we need to verify that the element we have given is unique, i.e. given \(\zeta : \nabla_0 X\) such that \(f z = \zeta\) for \(z : P\), we need to show \(\zeta = \prod_{z : P} f(z) \downarrow= x\). By applying \(\neg\neg\)-separatedness that we showed above, it suffices to derive the double negation of equality, which we can do using \(\neg\neg\)-dense propositions, as before.

  Finally we need to show that \(\nabla_0\) gives a reflective subuniverse in the sense of \cite[Definition 1.3]{rijkeshulmanspitters}. That is, given a map \(g : X \to Y\) where \(Y\) is a \(\neg\neg\)-separated \(\neg\neg\)-sheaf, we need to show that \(g\) extends uniquely along the unit \(\eta : X \to \nabla_0 X\). First note that any \(\neg \neg\)-separated type is a set by \cite[Corollary 7.2.3]{hottbook}, so we can also assume that \(Y\) is a set. Now any element of \(\nabla_{0} X\) is of the form \(\xi : X \to \Omega_{\neg\neg}\). Given \(x, x' : X\) such that \(\xi(x)\) and \(\xi(x')\), we have \(\neg \neg x = x'\), and so also \(\neg \neg g(x) = g(x')\). Since \(Y\) is \(\neg \neg\)-separated, we deduce that for such \(x, x'\) we have \(g(x) = g(x')\). Hence the function \(\sum_{x : X} \xi(x) \to Y\) given by restricting \(g\) is constant. Since \(Y\) is a set it follows that the map factors through the truncation \(\| \sum_{x : X} \xi(x) \to Y \|\). However, \(\| \sum_{x : X} \xi(x) \to Y \|\) is a \(\neg \neg\)-dense proposition, and so it extends to a unique element of \(Y\), from the assumption that \(Y\) is a \(\neg \neg\)-sheaf. By applying this to each fibre of the unit map \(\eta : X \to \nabla_{0} X\), we see that we can extend \(g\) uniquely to \(X\), as required.
\end{proof}

\begin{thm}
  A \(\neg\neg\)-sheaf is \(0\)-truncated precisely if it is \(\neg\neg\)-separated. Hence \(\nabla_0\) reflects onto \(\neg\neg\)-sheaves that are \(0\)-truncated.
\end{thm}

\begin{proof}
  As we already observed, any \(\neg\neg\)-separated type is automatically a set.

  For the converse, assume that \(X\) is a \(\neg \neg\)-sheaf and a set, and suppose we have \(x, x'\) such that \(\neg \neg x = x'\). Since \(X\) is a \(\neg \neg\)-sheaf, so is the identity type \(x = x'\). However, we can now see that \(x = x'\) is a proposition, a \(\neg \neg\)-sheaf and \(\neg \neg\)-dense. It follows directly that \(x = x'\) is inhabited. That is, we have deduced \(x = x'\) from \(\neg \neg x = x'\).
\end{proof}

We can define elements of \(\nabla_0 X\) by splitting into cases, the same way we would in classical logic, as we show in the following lemma.
\begin{lemma}
  \label{lem:defbycases}
  Given a proposition \(P\), and elements \(x, y\) of a type \(X\), we can define \(\xi : \nabla_0 X\) such that \(P\) implies \(\xi \downarrow= x\) and \(\neg P\) implies \(\xi \downarrow= y\).
\end{lemma}

\begin{proof}
  This follows formally from the fact that \(\nabla_{0}\) is a \(\neg\neg\)-sheaf. Namely, we define a map \(P + \neg P \to \nabla_{0} X\) sending \(\inl(\ast)\) to \(\eta(x)\) and \(\inr(\ast)\) to \(\eta(y)\). Since \(P + \neg P\) is a \(\neg \neg\)-dense proposition, this map extends uniquely to an element of \(\nabla_{0} X\).
\end{proof}

We will write elements of \(\nabla_0 X\) constructed using Lemma~\ref{lem:defbycases} using the usual notation for definition by cases, i.e. as follows.
\[
  \xi :=
  \begin{cases}
    x & P \\
    y & \neg P
  \end{cases}
\]

In particular Lemma~\ref{lem:defbycases} gives us one direction of the equivalence in the theorem below.
\begin{thm}
  \label{thm:twosoclassifier}
  The types \(\nabla_0 2\) and \(\Omega_{\neg \neg}\) are equivalent. Hence there is an equivalence between maps \(A \to \nabla_0 2\) and \(\neg\neg\)-stable subtypes of \(A\).
\end{thm}

\begin{proof}
  We explicitly describe the isomorphism as follows. Given \(p : \Omega_{\neg \neg}\), we map \(p\) to \(1\) if \(p\) is true and to \(1\) otherwise. We define the map \(\nabla_{0} 2 \to \Omega_{\neg\neg}\) by sending \(x : \nabla_{0} 2\) to the proposition \(x = \eta(1)\), which is \(\neg \neg\)-stable since \(\nabla_{0} 2\) is \(\neg \neg\)-separated. It is straightforward to check these two maps give an isomorphism.
\end{proof}

\subsection{Turing reducibility for modalities}
We recall from \cite{rijkeshulmanspitters} that there is a canonical way to order modalities.

\begin{prop}
  Suppose we are given modalities \(\bigcirc\) and \(\Diamond\). The following are equivalent.
  \begin{enumerate}
  \item If \(X\) is \(\bigcirc\)-connected, then it is \(\Diamond\)-connected.
  \item If \(X\) is \(\Diamond\)-modal, then it is \(\bigcirc\)-modal.
  \end{enumerate}
\end{prop}

\begin{proof}
  This follows formally from the theory of Galois connections, as in the proof of \cite[Theorem 3.25]{rijkeshulmanspitters}.
\end{proof}

\begin{defn}
  For modalities \(\bigcirc\) and \(\Diamond\), we write \(\bigcirc \leq_{T} \Diamond\) to mean that every \(\bigcirc\)-connected type is \(\Diamond\)-connected.
\end{defn}

It is clear by definition that this defines a partial order on the class of modalities, with antisymmetry following from Theorem~1.53 and Lemma~1.46 of \cite{rijkeshulmanspitters}.
\begin{rmk}
  We chose the ordering to agree with the usual ordering of Turing reducibility, which is the opposite to the one considered in \cite{rijkeshulmanspitters}.
\end{rmk}

\begin{prop}
  \label{prop:nullmodalityordering}
  Suppose that \(\bigcirc\) and \(\Diamond\) are modalities, and that \(\bigcirc\) is the nullification of a family of types \(B : A \to \univ\). Then \(\bigcirc \leq_{T} \Diamond\) iff for all \(a : A\), \(B(a)\) is \(\Diamond\)-connected.
\end{prop}

\begin{proof}
  First suppose that \(\bigcirc \leq_{T} \Diamond\). Note that for all \(a : A\), \(B(a)\) is \(\bigcirc\)-connected, by the definition of null types and \cite[Corollary 1.37]{rijkeshulmanspitters}. Hence each \(B(a)\) is also \(\Diamond\)-connected, by assumption.

  Now suppose that every \(B(a)\) is \(\Diamond\)-connected. It follows, again by \cite[Corollary 1.37]{rijkeshulmanspitters} that every \(\Diamond\)-modal type is \(B(a)\)-null for all \(a : A\), and thereby \(\bigcirc\)-modal, and so \(\bigcirc \leq_{T} \Diamond\).
\end{proof}

\subsection{Oracles and oracle modalities}
\label{sec:oracles}

Using \(\nabla_0\) we can now give our definitions of oracle and oracle modality.
\begin{defn}
  An \emph{oracle} is a map \(A \to \nabla_0 B\) for \(A, B : \univ\).
\end{defn}

The intuitive idea is that we can think of a function \(\chi : A \to \nabla_0 B\) as an external function \(A \to B\) in \(\sets\) that we are not able to compute. We are able to define non constructive functions by Lemma~\ref{lem:defbycases}, and can reason about the functions non constructively, as long as we stay inside the world of \(\neg\neg\)-sheaves. Note that by Theorem~\ref{thm:twosoclassifier} we can view any \(\neg\neg\)-stable subset of a type \(A\) as an oracle \(A \to \nabla_0 2\).

Also note that by the general properties of modalities, we have the following proposition.
\begin{prop}
  For \(A, B : \univ\) there is an equivalence between the type of oracles \(A \to \nabla_0 B\) and functions \(\nabla_0 A \to \nabla_0 B\).
\end{prop}

\begin{defn}
  Let \(\chi : A \to \nabla_0 B\) be an oracle. The \emph{oracle modality} associated to \(\chi\) is the nullification of the family of types \(a : A \vdash \chi (a) {\downarrow}\). We will write the modality as \(\bigcirc_\chi\).
\end{defn}

We think of elements of the type \(\bigcirc_\chi X\) as elements of \(X\) that can be computed using queries to the oracle \(\chi\). We have an inclusion function \(\eta : X \to \bigcirc_\chi X\) giving those elements of \(X\) that can be computed without any oracle queries. However, since \(\bigcirc_\chi X\) is null, we also know that any function \(\chi(a) {\downarrow} \to \bigcirc_\chi X\) extends uniquely to an element of \(\bigcirc_\chi X\). For example, we can in particular find for each \(a : A\) an element of \(\bigcirc_\chi (\chi(a) {\downarrow})\).

Based on this intuition we can expect that, given an element \(z : \bigcirc_\chi X\), it is possible to erase the information telling us how to compute it relative to \(\chi\) and just return the underlying element with no computational information, i.e. an element of \(\nabla_0 X\). We can do this as long as \(B\) is \(\neg\neg\)-separated, as follows.
\begin{thm}
  Let \(\chi : A \to \nabla_0 B\) be an oracle where \(B\) is \(\neg\neg\)-separated. Then for every \(X : \univ\) there is a unique map \(\mathtt{erase} : \bigcirc_\chi X \to \nabla_0 X\) making the following commutative triangle.
  \[
    \begin{tikzcd}
      X \ar[r] \ar[d] & \nabla_0 X \\
      \bigcirc_\chi X \ar[ur, swap, "\mathtt{erase}"] &
    \end{tikzcd}
  \]
\end{thm}

\begin{proof}
  By the general properties of modalities, it suffices to show that \(\nabla_0 X\) is \(\bigcirc_\chi\)-modal, i.e. that is null for the family \(a : A \vdash \chi (a){\downarrow}\). However, since \(B\) is \(\neg\neg\)-separated, \(\chi (a){\downarrow}\) is a proposition, and it is clearly \(\neg\neg\)-dense, so this follows from that fact that \(\nabla_0 X\) is a \(\neg\neg\)-sheaf.
\end{proof}

\begin{lemma}
  \label{lem:stripandrecall}
  Given an oracle \(\chi : A \to \nabla_0 B\) such that \(B\) is \(\neg\neg\)-separated, a type \(X\) and an element \(x\) of \(\bigcirc_\chi X\) then we can prove \(\bigcirc_\chi(\mathtt{erase}(x) {\downarrow})\).
\end{lemma}

\begin{proof}
  E.g. we can prove this directly by modal induction on \(\bigcirc_\chi X\), since \(\bigcirc_\chi(\mathtt{erase}(x) {\downarrow})\) is modal, and it is clear for elements of the form \(\eta(x)\).
\end{proof}

The assumption that \(B\) is \(\neg\neg\)-separated also allows us to show \(\bigcirc_\chi\) is lex.
\begin{thm}
  Let \(\chi : A \to \nabla_0 B\) where \(B\) is \(\neg\neg\)-separated. Then \(\bigcirc_\chi\) is a topological modality, and in particular is lex.
\end{thm}

\begin{proof}
  \(\bigcirc_\chi\) is the nullification of the family \(a : A \vdash \chi (a){\downarrow}\). Since \(B\) is \(\neg\neg\)-separated, each \(\chi (a){\downarrow}\) is a proposition, and so we can apply \cite[Corollary 3.12]{rijkeshulmanspitters}.
\end{proof}

\begin{defn}
  Suppose we are given oracles \(\chi : A \to \nabla_0 B\) and \(\chi' : A' \to \nabla_0 B'\). We say \(\chi\) is \emph{Turing reducible} to \(\chi'\) and write \(\chi \leq_T \chi'\) if \(\bigcirc_{\chi} \leq_{T} \bigcirc_{\chi'}\).
\end{defn}

It is clear that this defines a preorder on oracles. We say two oracles \(\chi, \chi'\) are \emph{Turing equivalent} if \(\bigcirc_{\chi} = \bigcirc_{\chi'}\).

Also note that we have the following proposition.
\begin{prop}
  Let \(B\) be \(\neg\neg\)-separated. An oracle \(\chi : A \to \nabla_0 B\) extends to a total function \(A \to B\) if and only if \(\bigcirc_\chi\) is the identity modality.
\end{prop}

\begin{proof}
  First note that to extend \(\chi : A \to \nabla_0 B\) to a total function is precisely to give a term of type \(\prod_{a : A} \sum_{b : B} \chi(a) = \iota(b)\). It is straightforward to check that \(\chi(a) = \iota(b)\) is equivalent to \(\chi(a) \downarrow= b\), and so since \(B\) is \(\neg\neg\)-separated \(\chi(a) {\downarrow}\) is a proposition and so we have such a term precisely when \(\chi(a) {\downarrow}\) is contractible for all \(a\). If every \(\chi(a) {\downarrow}\) is contractible, then every type is \(\bigcirc_\chi\)-modal, and so the reflective subuniverse corresponding to \(\bigcirc_\chi\) is the trivial one. Conversely, we have that  every \(\chi(a) {\downarrow}\) is \(\bigcirc_\chi\)-connected in any case, and when \(\bigcirc_\chi\) is the identity this precisely says \(\chi(a) {\downarrow}\) is contractible.
\end{proof}

% FINISH AND ADD BACK IF TIME
% \begin{thm}
%   Let \(\chi, \chi' : A \to \nabla B\) be oracles where \(B\) is \(\neg\neg\)-separated. Then \(\chi \leq_T \chi'\) if and only if 
% \end{thm}

\begin{defn}
 We define the \emph{Turing degrees} to be the collection of modalities of the form \(\bigcirc_\chi\) for \(\chi : \NN \to \nabla_0 2\). Equivalently, it is the set quotient of oracles \(\NN \to \nabla_0 2\) by Turing equivalence.
\end{defn}

As a sanity check, we verify that this notion is non trivial.
\begin{prop}
  \label{prop:haltingset}
  Assume Church's thesis. Then there is an oracle \(\chi : \NN \to \nabla_0 2\) such that \(\bigcirc_\chi\) is non trivial.
\end{prop}

\begin{proof}
  Using Lemma \ref{lem:defbycases} we can define the halting set, i.e.
  \[
    \chi(n) :=
    \begin{cases}
      0 & \neg \;\varphi_e(e) {\downarrow} \\
      1 - \varphi_e(e) & \varphi_e(e) {\downarrow}
    \end{cases}
  \]
  By the usual diagonal argument \(\chi\) cannot extend to a computable function \(\NN \to 2\). Hence by \(\ect\) it cannot extend to \emph{any} function \(\NN \to 2\). We have \(\prod_{n : \NN} \bigcirc_\chi (\chi(n) {\downarrow})\), and so if \(\bigcirc_\chi\) was the identity we could deduce \(\prod_{n : \NN} \chi(n) {\downarrow}\) giving a contradiction.
\end{proof}

% We illustrate that this is non trivial, as follows.

% \begin{defn}
%   The \emph{complement} of an oracle \(\chi : \NN \to \nabla 2\) is the oracle \(\neg \chi : \NN \to \nabla 2\) defined by taking \(\neg \chi(n)\) to be \(\top\) when \(\chi(n) = \bot\) and to be \(\bot\) when \(\chi(n) = \top\).
% \end{defn}

% \begin{prop}
%   \label{prop:complement}
%   For any \(\chi : \NN \to \nabla 2\) we have \(\chi \equiv_T \neg \chi\).
% \end{prop}

% \begin{cor}
%   There are \(\chi, \chi' : \NN \to \nabla 2\) such that \(\chi\) and \(\chi'\) are Turing equivalent, but not many-one equivalent.
% \end{cor}

% \begin{proof}
%   Let \(\chi\) be the set of Turing machines \(e\) that halt on input \(e\) and \(\chi' := \neg \chi\). Then \(\chi\) and \(\chi'\) are Turing equivalent by Proposition \ref{prop:complement}, but \(\chi'\) is not computably enumerable and so is not many-one reducible from \(\chi\).
% \end{proof}

\section{Comparison with weak truth table reducibility}
\label{sec:comparison-with-weak}

\newcommand{\leqwtt}{\leq_{\mathrm{wtt}}}
\newcommand{\lst}{\mathtt{List}}

To illustrate how our formulation of Turing reducibility fits in with other notions of reducibility, in this section we compare it with a definition of weak truth table reducibility (wtt reducibility) along similar lines to truth table reducibility, as defined in \cite{forsterjahn}. We show that Turing reducibility is strictly weaker than wtt reducibility, i.e. wtt reducibility implies Turing, and we give an example of a pair of oracles that are Turing reducible but not wtt reducible. We point out that Forster, Kirst and M\"{u}ck already showed that their synthetic definition of Turing reducibility is strictly weaker than truth table reducibility \cite{forsterkirstmuck}.

\begin{defn}
  For oracles \(\chi : A \to \nabla_0 B\) and \(\chi': A' \to \nabla_0 B'\), we say \(\chi\) is \emph{weak truth table reducible} to \(\chi'\) and write \(\chi \leqwtt \chi'\) if there is a function \(f : A \to \lst A'\) such that for all \(x : A\), if \(\chi'(y) {\downarrow}\) for every element \(y\) of \(f(x)\) then \(\chi(x) {\downarrow}\).
\end{defn}

For oracles \(\NN \to \nabla_0 B\) we can also the following more convenient form, which is straightforward to prove.
\begin{prop}
  For oracles \(\chi : \NN \to \nabla_0 B\) and \(\chi': \NN \to \nabla_0 B'\), \(\chi \leqwtt \chi'\) if and only if there is a function \(f : \NN \to \NN\) such that for all \(n : \NN\), if \(\chi'(m) {\downarrow}\) for all \(m < f(n)\) then \(f(n) {\downarrow}\).
\end{prop}

\begin{thm}
  For oracles \(\chi : A \to \nabla_0 B\) and \(\chi': A' \to \nabla_0 B'\), if \(\chi \leqwtt \chi'\) then \(\chi \leq_T \chi'\).
\end{thm}

\begin{proof}
  Let \(x : A\). We have \(\prod_{y \in f(x)} \chi'(y) {\downarrow} \to \chi(x) {\downarrow}\). Since modalities preserve finite products \cite[Lemma 1.27]{rijkeshulmanspitters} and are functorial \cite[Lemma 1.21]{rijkeshulmanspitters}, we can deduce \(\prod_{y \in f(x)} \bigcirc_{\chi'} \chi'(y) {\downarrow} \,\to\, \bigcirc_{\chi'} \chi(x) {\downarrow}\). However, we can find an element of \(\prod_{y \in f(x)} \bigcirc_{\chi'} \chi'(y) {\downarrow}\) directly, and so deduce \(\bigcirc_{\chi'} \chi(x) {\downarrow}\), confirming \(\chi \leq_T \chi'\).
\end{proof}

Below we will assuming we have an encoding of finite lists of numbers \(\lst(\NN) \cong \NN\). We write \(\langle k_1,\ldots,k_n \rangle\) for the natural number encoding the list \([k_1,\ldots,k_n]\).
\begin{lemma}
  \label{lem:wttrepr}
  Assume \(\chi \leqwtt \chi'\) and assume \(\ect\). Then there merely exists a total computable function \(\varphi_{e_0} : \NN \to \NN\) and a partial computable function \(\varphi_{e_1} : \NN \rightharpoondown \NN\) satisfying the following. For all \(n : \NN\), \(\chi(n)\) and \(\varphi_{e_1}(\langle n, \chi'(0),\ldots,\chi'(\varphi_{e_0}(n) - 1) \rangle)\) are both defined and are equal as elements of \(2\).
\end{lemma}

\begin{proof}
  Applying Church's thesis gives us \(e_0\) such that for all \(n : \NN\),  \(\prod_{m < \varphi_{e_0}(n)} \chi'(m)\) implies \(\chi(n) {\downarrow}\). We can now apply \(\ect\) to find the required \(e_1\).
\end{proof}

Our example follows Rogers \cite[Section 9.1]{rogers}.
\begin{thm}
  There are oracles \(\zeta, \kappa : \NN \to \nabla_0 2\) such that \(\zeta \leq_T \kappa\) and \(\zeta \nleq_\mathrm{wtt} \kappa\).
\end{thm}

\begin{proof}
  We first define \(\kappa\) to encode the halting problem for \(\varphi_e(n)\) for all \(e, n : \NN\), i.e.
  \begin{equation*}
    \kappa(\langle e , n \rangle) =
    \begin{cases}
      \top & \varphi_e(n) {\downarrow} \\
      \bot & \text{otherwise}
    \end{cases}
  \end{equation*}

  We now construct \(\zeta\).

  We first observe that for all \(e_0, e_1, n : \NN\) we can show the following, by repeatedly querying the oracle \(\kappa\).
  \[\bigcirc_\kappa (\dec( \varphi_{e_0}(n) {\downarrow} \times \varphi_{e_1}(\langle n, \kappa(0), \ldots, \kappa(\varphi_{e_0}(n) - 1)\rangle) \downarrow= \top)\]
  Next, by reasoning internally in the reflective subuniverse for \(\bigcirc_\kappa\), we can deduce the following for all \(e_0, e_1 : \NN\).
  \begin{multline*}
    \bigcirc_\kappa\sum_{b \in 2} (b = \bot \;\leftrightarrow\ \\
    \varphi_{e_0}(\langle e_0, e_1 \rangle) {\downarrow} \times \varphi_{e_1}(\langle \langle e_0, e_1 \rangle, \kappa(0), \ldots, \kappa(\varphi_{e_0}(\langle e_0, e_1 \rangle) - 1)\rangle) \downarrow= \top)
  \end{multline*}
  Composing with relativised projection to \(2\) and \(\mathtt{erase} : \bigcirc_\kappa 2 \to \nabla_0 2\) gives us \(\zeta : \NN \to \nabla_0 2\), which is Turing reducible to \(\kappa\) by construction.

  Finally, suppose that we had \(\zeta \leqwtt \kappa\). By Lemma~\ref{lem:wttrepr} there merely exists \(e_0, e_1\) such that for all \(n\), \(\zeta(n)\) and \(\varphi_{e_1}(\langle n, \kappa(0),\ldots,\kappa(\varphi_{e_0}(n) - 1) \rangle)\) are both defined and are equal as elements of \(2\). But this gives a contradiction for \(n = \langle e_0, e_1 \rangle\).
\end{proof}

\begin{rmk}
  The results of this section appear in {\tt WeakTT.agda} in the formalisation.
\end{rmk}

\section{Markov induction}
\label{sec:markov-induction}

\subsection{Markov induction and Markov's principle in HoTT}
\label{sec:mark-induct-mark}

Recall that Markov's principle is a commonly used axiom in recursive constructive mathematics that asserts the existence of numbers that can be found by unbounded search, possibly using the law of excluded middle. Formally, it is stated as follows.
\begin{defn}
  \emph{Markov's principle}, \(\markov\), states that given a family of types \(P : \NN \to \univ\) together with a proof of \(\prod_{n : \NN} \dec(P(n))\) we have \(\neg\neg \sum_{n : \NN} P(n) \to \sum_{n : \NN} P(n)\).
\end{defn}

For our purposes, however, \(\markov\) is too weak. The reason is that it requires as antecedent a proof that every \(P(n)\) is decidable, which in realizability semantics requires a computable function that tells us whether or not each \(P(n)\) is true. We will give several examples of arguments where it is only possible to compute this using an oracle (i.e. \(f : \NN \to \bigcirc_\chi (P(n) + \neg P(n))\)) and still wish to carry out an unbounded search. We will derive the necessary relativised version of \(\markov\) from an axiom that we denote \emph{Markov induction}. Although the precise formulation and applications in synthetic computability theory are new, Markov induction is equivalent to the axiom of Generalized Markov's principle for double negation, as formulated by Hofmann, Van Oosten and Streicher in \cite{hofmannvoostenstreicher}. Essentially, we can extend the binary relation ``is successor of'' from \(\NN\) to a binary relation on \(\nabla_{0} \NN\), defined below. Markov induction says that this binary relation is well founded in the sense of \cite[Section 10.3]{hottbook}. We state the axiom as an induction principle for convenience both in checking that it holds in cubical assemblies, and in using it in practice.

\begin{defn}
  Let \(\mu, \nu : \nabla_0 \NN\). We say \(\nu\) is a \emph{successor} of \(\mu\) if \(\mu \downarrow = m\) implies \(\nu \downarrow = m + 1\) for all \(m : \NN\). We write \(\issuc(\mu, \nu)\).
\end{defn}

\begin{axiom}
  \emph{Markov induction}, \(\markovind\) states that given \(P : \nabla_0 \NN \to \univ\) together with a term \(h : \prod_{\nu : \nabla_0 \NN} (\prod_{\mu : \nabla_0 \NN} \issuc(\mu, \nu) \to P(\mu)) \to P(\nu)\), we have a section \(s : \prod_{\nu : \nabla_0 \NN} P(\nu)\), together with paths witnessing the following computation rule: for all \(\mu, \nu : \nabla_{0} \NN\) if \(\issuc(\mu, \nu)\) we have \(s(\nu) = h (\nu, \lambda \mu, x.s(\mu))\).
\end{axiom}

\begin{rmk}
  The computation rule was added to make the axiom more natural as an induction principle, but isn't required for any of the results of the paper and isn't postulated in the Agda formalisation.
\end{rmk}

\begin{thm}[Relativised Markov's principle]
  \label{thm:relmarkov}
  Let \(\chi : A \to \nabla_0 B\) be an oracle, a family of types \(P : \NN \to \univ\) and \(f :
  \prod_{n : \NN} \bigcirc_\chi \dec(P(n))\). Then \(\neg \neg \bigcirc_\chi (\sum_{n \in \NN} P(n)) \;\to\; \bigcirc_\chi (\sum_{n \in \NN} P(n))\).
\end{thm}

\begin{proof}
  We first show by \(\markovind\) that for all \(\nu : \nabla_0 \NN\), we have \(\prod_{n : \NN} \bigcirc_\chi \dec(\nu \downarrow= n) \,\to\, \bigcirc_\chi (\nu {\downarrow})\). Suppose that the statement holds for \(\mu\) where \(\nu\) is a successor of \(\mu\) and that \(\prod_{n : \NN} \bigcirc_\chi \dec(\nu \downarrow= n)\). Internal to the modality \(\bigcirc_\chi\), we split into two cases depending on whether \(\nu \downarrow= 0\). If \(\nu \downarrow= 0\), then we are done. If not, then we can show \(\nu\) is the successor of a unique \(\mu : \nabla_0 \NN\). Since \(\mu \downarrow= n\) if and only if \(\nu \downarrow= n + 1\) we have \(\prod_{n : \NN} \bigcirc_\chi \dec(\mu \downarrow= n)\), so \(\bigcirc_\chi (\mu {\downarrow})\) by the inductive hypothesis, and so \(\bigcirc_\chi (\nu {\downarrow})\).

  Finally, given \(P : \NN \to \univ\) and \(f : \prod_{n : \NN} \bigcirc_\chi \dec(P(n))\), we apply the above with \(\nu\) such that \(\nu \downarrow= n\) when \(n\) is least such that \(P(n)\). To do this, we need to show that we can decide, internally in \(\bigcirc_{\chi}\)-modal types, whether a given \(n\) is least such that \(P(n)\) is true. However, this is simply an internal version of the standard fact that if \(P(n)\) is decidable for all \(n\), then it is also decidable whether a given \(n\) is least such that \(P(n)\).
\end{proof}

\begin{thm}
  \label{thm:markovindinternal}
  If \(\markovind\) holds in HoTT, then it also holds in the reflective subuniverse of \(\bigcirc_{\chi}\)-modal types, for any oracle \(\chi : A \to \nabla_{0} B\).
\end{thm}

\begin{proof}
  Since every \(\bigcirc_{\chi}\)-connected map is \(\neg\neg\)-dense, the unit map \(\NN \to \bigcirc_{\chi} \NN\) must be \(\neg\neg\)-dense, and so the composition \(\NN \to \bigcirc_{\chi} \NN \to \nabla_{0} \bigcirc_{\chi} \NN\) must be \(\nabla_{0}\)-connected. By uniqueness of factorisations for orthogonal factorisation systems, we can deduce that \(\nabla_{0}\bigcirc_{\chi} \NN \simeq \nabla_{0} \NN\). Note furthermore that the definition of \(\neg\neg\)-sheaf is the same according to the internal logic of \(\bigcirc_{\chi}\)-modal types as externally and that \(\bigcirc_{\chi} \NN\) is the natural number type in the reflective subuniverse for \(\bigcirc_{\chi}\). It follows that applying \(\nabla_{0}\) to the natural number type inside \(\bigcirc_{\chi}\)-modal types is the same as \(\nabla_{0} \NN\) externally.

  Finally, note that the statement \(\prod_{\nu : \nabla_0 \NN} (\prod_{\mu : \nabla_0 \NN} \issuc(\mu, \nu) \to P(\mu)) \to P(\nu)\) is also absolute, by the same reasoning as above, and so we can apply external Markov induction to show every instance of the version internal to \(\bigcirc_{\chi}\).
\end{proof}

\begin{rmk}
  The results of this section appear in the file {\tt OracleModality.agda} of the formalisation.
\end{rmk}

\subsection{Markov induction in cubical assemblies}
\label{sec:markov-induction-1}

In this section we will justify Markov induction by showing it holds in realizability models. As part of this we show how to give the axiom computational meaning using the fixed point theorem.

\begin{lemma}
  \label{lem:uniformn}
  In \(\asm\), \(U(\NN)\) is the \(\neg\neg\)-sheafification of the natural number type, assuming classical logic in the metatheory.
\end{lemma}

\begin{proof}
  See e.g. \cite[Section 2.6.3]{vanoosten}.
\end{proof}

\begin{thm}
  \label{thm:miasm}
  Markov induction holds in the locally cartesian closed category of assemblies, assuming classical logic in our metatheory.
\end{thm}

\begin{proof}
  By Lemma~\ref{lem:uniformn} we can replace \(\nabla_0 \NN\) with \(U(\NN)\) in the statement of Markov induction.
  Suppose that we are working in a context \(B\), and we have a family of types \(P \to B \times U \NN\) and for some \(b \in | B |\) a map \(f : (\prod_{\mu : U \NN} \issuc(\mu, \nu) \to P(b, \mu)) \to P(b, \nu)\). Note that the function underlying the map \(f\) gives us a choice of element \(p_0 : P(b, 0)\), since we vacuously have a unique element of \(\prod_{\mu : U \NN} \issuc(\mu, \nu) \to P(\mu)\). We also have functions \(h(n) : P(b, n) \to P(b, n + 1)\) for each \(n\). This uniquely determines, by recursion on \(\NN\) in our metatheory,\footnote{i.e. in \(\sets\) rather than \(\asm\)} a section \(s\) of \(| P | \to | B | \times \NN\) such that \(s(0) = p_0\) and \(s(n + 1) = h(s(n))\). Note that this construction is automatically stable under pullback. We only need to show how to define a realizer for \(s\) given a realizer for \(b\).

  Suppose that \(d\) is a realizer for \(f\), as above. By the fixed point theorem\footnote{See e.g. \cite[Proposition 1.3.4(ii)]{vanoosten}.} we can find, uniformly in \(d\), a number \(e\) such that \(e 0 \simeq d (\lambda x, y.e 0)\). We verify by external induction on \(\NN\) that \(e\) is a realizer for \(s\), i.e. \(e 0 \Vdash s(n)\) for all \(n\).

  First suppose \(n = 0\). Note that \(\lambda x, y.e 0\) is vacuously a realizer for \(\prod_{\mu : U \NN} \issuc(\mu, 0) \to P(b, \mu)\). Hence \(d (\lambda x, y.e 0)\) is a realizer for \(p_0\) and in particular \(d (\lambda x, y.e 0){\downarrow}\).

  Next suppose \(n = m + 1\). By the inductive hypothesis, \(e 0 \Vdash s(m)\). Hence \(\lambda x, y.e 0 \Vdash \prod_{\mu : U \NN} \issuc(\mu, n) \to P(b, \mu)\). We deduce that \(d (\lambda x, y.e 0)\) is defined and realizes \(s(n)\). It follows that the same is true for \(e 0\).

  Since we can find such an \(e\) uniformly in \(d\), this verifies Markov induction.
\end{proof}

\begin{lemma}
  \label{lem:deltapresnabn}
  In \(\cubasm\), \(\nabla_0 \NN\) is isomorphic to \(\Delta U(\NN)\).
\end{lemma}

\begin{proof}
  Observe that Definition \ref{defn:dnsheafification} uses only the classifier for \(\neg\neg\)-propositions, the standard type formers of dependent sums, products and identity types, and the natural number type. However, all of this is preserved by \(\Delta\).\footnote{Recall Theorem \ref{thm:deltpresdnclass} for the case of the classifier for \(\neg\neg\)-propositions.}
\end{proof}

\begin{thm}
  Markov induction holds in the interpretation of cubical type theory in cubical assemblies.
\end{thm}

\begin{proof}
  By Lemma~\ref{lem:deltapresnabn} we may replace \(\nabla_0 \NN\) with \(\Delta U(\NN)\) in the statement of Markov induction.
  Given a fibration \(P \rightarrow B \times \Delta(\nabla_0 \NN)\) together with a natural transformation \(f : (\prod_{\mu : \nabla_0 \NN} \issuc(\mu, \nu) \to P(b, \mu)) \to P(b, \nu)\), for each cube \([k] \in \square\) we can restrict \(P\) and \(f\) to get \(P_k \to \nabla_0 \NN\) and \(f_k : (\prod_{\mu : \nabla_0 \NN} \issuc(\mu, \nu) \to P_k(\mu)) \to P_k(\nu)\). Applying \(\markovind\) in \(\asm\) gives us a family of sections \(s_k\) for each \(P_k\). Furthermore, as remarked in the proof of Theorem~\ref{thm:miasm}, we can ensure the \(s_k\) are uniquely determined, and so they do give us a natural transformation \(s\), which is a section of \(P \to B \times \Delta(U \NN)\).
\end{proof}

\subsection{Markov induction in synthetic computability theory}
\label{sec:some-synth-comp}

Below we give a few simple examples to illustrate the use of Markov induction via relativised Markov's principle in synthetic computability theory. In each case we need to construct a function computable from an oracle that at some point carries out an unbounded search. Markov induction allows us to use non constructive reasoning to show that this gives a well defined function that is computable from the oracle. In particular, under the appropriate conditions we can assume that particular elements of \(\nabla_0 \NN\) are defined. Since we are working synthetically, we simultaneously construct the function and show that it has a desired property.

\begin{defn}
  If \(\chi : A \to \nabla_0 B\), the \emph{graph} of \(\chi\) is the function \(\overline{\chi} : A \times B \to \nabla_0 2\) defined as below:
  \begin{displaymath}
    \overline{\chi}(a, b) :=
    \begin{cases}
      1 & \chi(a) \downarrow= b \\
      0 & \text{otherwise}
    \end{cases}
  \end{displaymath}
\end{defn}

\begin{thm}
  For \(\chi : \NN \to \nabla_0 \NN\) we have \(\chi \equiv_T \overline{\chi}\).
\end{thm}

\begin{proof}
  It is straightforward to show \(\overline{\chi} \leq_T \chi\): Given \(n, m : \NN\), we have \(\bigcirc_\chi \chi(n) {\downarrow}\). Working internally to \(\bigcirc_\chi\)-modal types, we can decide whether \(\chi(n) \downarrow= m\). If it does, then \(\overline{\chi}(n, m) \downarrow= 1\). Otherwise, \(\overline{\chi}(n, m) \downarrow= 0\).

  We now show that \(\chi \leq_T \overline{\chi}\). Suppose that we are given \(n : \NN\). We will use relativised Markov's principle, Theorem~\ref{thm:relmarkov}, to find an element of \(\bigcirc_{\overline{\chi}} (\sum_{m : \NN} (\chi(n) \downarrow= m))\). It suffices to show \(\bigcirc_{\overline{\chi}} \dec(\chi(n) \downarrow= m)\) and \(\neg\neg (\bigcirc_{\overline{\chi}} (\sum_{m : \NN} \chi(n) \downarrow= m))\). For the former, given \(m : \NN\) we have \(\bigcirc_{\overline{\chi}}\overline{\chi}(n, m) {\downarrow}\). Working internally in \(\bigcirc_{\overline{\chi}}\)-modal types, we split into two cases depending on whether \(\overline{\chi}(n, m) \downarrow= 1\) or \(\overline{\chi}(n, m) \downarrow= 0\). In the former case we have \(\chi(n) \downarrow= m\) and in the latter case we have \(\neg (\chi(n) \downarrow= m)\), as required. Finally, we have by the definition of \(\nabla_0\) the double negation of \(\chi(n) {\downarrow}\), which directly gives us \(\neg\neg (\bigcirc_{\overline{\chi}} (\sum_{m : \NN} \chi(n) \downarrow= m))\).
\end{proof}

Although we defined oracles to be functions \(\NN \to \nabla_0 2\), following a common convention in computability theory, we can see from the above theorem that in fact we can consider all oracles \(\NN \to \nabla_0 \NN\) and obtain the same class of modalities, and so same definition of Turing degree.

\begin{thm}
  Every surjective function \(f : \nabla_0 \NN \to \nabla_0 \NN\) has a section \(s : \nabla_0 \NN \to \nabla_0 \NN\) with \(s \leq_T f\). If \(f\) is a bijection with inverse \(f^{-1}\), then \(f^{-1} \equiv_T f\).
\end{thm}

\begin{proof}
  First, since \(\nabla_0\) is a modality, every map \(\NN \to \nabla_0 \NN\) extends uniquely to a map \(\nabla_0 \NN \to \nabla_0 \NN\). It follows that it suffices to construct a map \(s : \NN \to \nabla_0 \NN\) such that for all \(\NN\) we have \(f(s(n)) = \iota (n)\), and that \(s\) is Turing reducible to \(f\). By Lemma~\ref{lem:stripandrecall} it suffices to construct for each \(n : \NN\) an element of \(\bigcirc_f (\sum_{m : \NN} f(\iota(m)) \downarrow= n)\).  We will do this using \(\markov_f\) taking \(P(m) := f(\iota(m))\). To do this, we need to check both \(\prod_{m : \NN} \bigcirc_f \dec(f(\iota(m)) \downarrow= n)\) and \(\neg\neg \sum_{m : \NN} f(\iota(m)) = n\). For the former, first note that we can show \(\bigcirc_f f(\iota(m)) {\downarrow}\). That is, internally in \(\bigcirc_f\)-modal types, we can assume \(f(\iota(m))\) denotes a natural number, say \(f(\iota(m)) \downarrow= l\). Still working internally, we can use decidable equality for \(\NN\) to split into 2 cases depending on whether \(l = n\). If it does, then we are done. If not, we must have \(\neg (f(\iota(m)) \downarrow= n)\) since \(f(\iota(m))\) cannot denote two distinct elements. Finally, to check \(\neg\neg \sum_{m : \NN} f(\iota(m)) = n\), since \(f\) is surjective, there merely exists \(\mu : \nabla_0 \NN\) such that \(f(\mu) = \iota(n)\). It follows that \(f(\mu) \downarrow= n\). By the construction of \(\nabla_0 \NN\) we can show the double negation of the statement that there is \(m\) such that \(\mu = \iota(m)\). From these we derive the double negation of \(\sum_{m : \NN} f(\iota(m)) = n\), as required.
\end{proof}

We next give some useful lemmas extending relativised Markov's principle.

\begin{lemma}[``Parallel search'']
  \label{lem:parallelsearch}
  Suppose we are given two families of types \(P, Q : \NN \to \univ\), such that both are decidable relative to \(\chi\) and \(\neg \neg (\sum_{n : \NN} P n \,+\, \sum_{n : \NN} Q n)\). Then we can deduce \(\bigcirc_\chi (\sum_{n : \NN} P n \,+\, \sum_{n : \NN} Q n)\).
\end{lemma}

\begin{proof}
  We define \(R : \NN \to \univ\) by \(R(n) := P(n) + Q(n)\). From the decidability of \(P\) and \(Q\) relative to \(\chi\), we deduce that \(R\) is decidable relative to \(\chi\). We clearly have \(\sum_{n : \NN} P n \,+\, \sum_{n : \NN} Q n \to \sum_{n : \NN} R(n)\), and so we can deduce \(\neg\neg \sum_{n : \NN} R(n)\) from our assumptions. We then apply \(\markov_\chi\) to deduce from this \(\sum_{n : \NN} R(n)\), which implies \(\sum_{n : \NN} P n \,+\, \sum_{n : \NN} Q n\).
\end{proof}

\begin{lemma}
  \label{lem:distinguish0}
  Let \(\chi : A \to \nabla_0 B\) be an oracle where \(B\) is \(\neg\neg\)-separated. Suppose we are given \(f, g : \NN \to \bigcirc_\chi X\) such that \(X\) is discrete, and \(f \neq g\). Then we can show \(\bigcirc_\chi (\sum_{n : \NN}  f(n) \neq g(n))\).
\end{lemma}

\begin{proof}
  We apply \(\markov_\chi\) with \(P(n) := f(n) \neq g(n)\). To do this we need to show that for all \(w, z : \bigcirc_\chi X\) we have \(\bigcirc_\chi \dec(w = z)\), which follows from the fact that \(X\) is discrete and \(\bigcirc_\chi\) is lex.
\end{proof}

Finally we give a technical lemma generalising Lemma~\ref{lem:distinguish0} that will be used later in Section~\ref{sec:effect-non-comp}.
\begin{lemma}
  \label{lem:distinguish} Suppose we are given \(f, g, h, k : \NN \to \bigcirc_\chi \NN\) such that \(\neg (f = g \times h = k)\). Then we can show \(\bigcirc_\chi (f \neq g + h \neq k)\).
\end{lemma}

\begin{proof}
  We apply Lemma~\ref{lem:parallelsearch} with \(P(n) := f(n) \neq g(n)\) and \(Q(n) := h(n) \neq k(n)\), and using the same argument as in Lemma~\ref{lem:distinguish0}
\end{proof}

\begin{rmk}
  Variants of the above lemmas are proved formally in the files {\tt OracleModalities.agda} and {\tt ParallelSearch.agda} of the Agda formalisation.
\end{rmk}

\section{The axiom of computable choice}
\label{sec:axiom-comp-choice}

We give a strengthened version of Church's thesis, which will be a synthetic version of Hyland's axiom of generalised Church's thesis \cite[Proposition 10.4]{hylandeff}.\footnote{We leave it as an exercise for the reader to prove generalised Church's thesis is equivalent to a statement of the same form as computable choice.} We will call the axiom \emph{computable choice} to emphasise that it implies some special cases of countable choice and to draw an analogy with the axiom of continuous choice \cite[Section XVI.1.3]{beeson85} in intuitionism.
\begin{axiom}
  The axiom of \emph{computable choice} states the following. As for extended Church's thesis we consider a synthetic version, postulating the existence of functions \(\varphi^{k}_{e}\), and using these to define partial functions \(\varphi_{e}\). Suppose we have \(R : \NN \times \NN \to \Omegann\) with \(\neg\neg\)-stable domain. That is,
  \begin{equation}
    \label{eq:compchoicehyp}
    \prod_{n : \NN} \left[\neg \neg \sum_{m : \NN} R n m \;\to\; \left\| \sum_{m : \NN} R n m \right\|\right].
  \end{equation}
  Then there merely exists \(e : \NN\) such that for all \(n : \NN\), if \(\neg\neg \sum_{n : \NN} R n m\) then we have \(\varphi_{e}(n) {\downarrow}\) and \(R n \varphi_{e}(n)\).
\end{axiom}

Note that by allowing arbitrary relations \(R : \NN \times \NN \to \Omegann\) this is stronger than synthetic Church's thesis in two ways:
\begin{enumerate}
\item It allows partial functions, as long as they have \(\neg\neg\)-stable domain, generalising Extended Church's thesis, which can be recovered as the special case where \(R\) is at most single valued.
\item It allows multivalued functions, and so also implies a useful special case of the axiom of choice.
\end{enumerate}

Having a weak form of choice available is especially useful when working with modalities generated by nullification, since the path constructors in the HIT definition can be thought of as a form of quotienting, and so choice can help us deal with maps into the nullification of a type. This will be important in Section~\ref{sec:relat-churchs-thes}.

\begin{prop}
  The axiom of computable choice holds in assemblies.
\end{prop}

\begin{proof}
  Similar to \cite[Proposition 10.4]{hylandeff}.
\end{proof}

\begin{thm}
  There is a reflective subuniverse of cubical assemblies that satisfies computable choice.
\end{thm}

\begin{proof}
  In the category of assemblies, we define two objects \(X\) and \(Y\) using the internal language for regular locally cartesian closed categories.
  \begin{align*}
    X &:= \{ (R, e) \in \Omegann^{\NN \times \NN} \times \NN \;|\; \forall_{n : \NN} \neg\neg \exists_{m : \NN} R n m \,\to\, \varphi_{e}(n) {\downarrow} \,\wedge\, R n \varphi_{e(n)} \} \\
    Y &:= \{ R \in \Omegann^{\NN \times \NN} \;|\; \forall_{n : \NN} \neg\neg \exists_{m : \NN} R n m \,\to\, \exists_{m : \NN} R n m \}
  \end{align*}
  Note that we have a canonical map \(s : X \to Y\) and that computable choice precisely states that \(s\) is surjective. In cubical assemblies we now consider the map \(\Delta(s) : \Delta(X) \to \Delta(Y)\). Using the techniques of \cite{uemuraswan} we can construct a reflective subuniverse of cubical assemblies such that \(\Delta(s)\) is homotopy surjective: note that the arguments in \cite[Section 6]{uemuraswan} go through whenever \(A \vdash B(a)\) is a family of discrete types in the locally cartesian closed category of cubical assemblies, not necessarily definable in type theory.

  Next note that \(X\) is constructed using only type operators that are preserved by \(\Delta\). This ensures that \(\Delta(X)\) agrees with the corresponding object constructed in cubical assemblies, i.e.
  \begin{equation*}
    \Delta(X) \cong \sum_{R : \NN \times \NN \to \Omegann} \sum_{e : \NN} \prod_{n : \NN} \neg\neg \sum_{m : \NN} R n m \,\to\, \varphi_{e}(n) {\downarrow} \,\wedge\, R n \varphi_{e(n)}
  \end{equation*}

  On the other hand, the definition of \(Y\) includes an existential quantifier, which makes it definitely not the same as the corresponding object \(Y'\) defined using propositional truncation in the model of HoTT in cubical assemblies.
  \begin{equation*}
    Y' := \sum_{R : \NN \times \NN \to \Omegann} \prod_{n : \NN} \neg\neg \sum_{m : \NN} R n m \,\to\, \left\| \sum_{m : \NN} R n m \right\|
  \end{equation*}
  However, we can factor \(\Delta(s)\) through \(Y'\), and the factorisation commutes with the projection to \(\Omegann^{\NN \times \NN}\).
  \begin{equation*}
    \begin{tikzcd}
      \Delta(X) \ar[dd, "\Delta(s)"] \ar[dr, "s'"] & \\
      & Y' \ar[dl, "t"] \dar["p"] \\
      \Delta(Y) \rar & \Omegann^{\NN \times \NN}
    \end{tikzcd}
  \end{equation*}
  We can show, internally in HoTT that the projection map \(p : Y' \to \Omegann^{\NN \times \NN}\) is injective. It follows that the map \(t\) must be homotopy injective. However, in the reflective subuniverse, we now have that \(\Delta(s) = t \circ s'\) where \(\Delta(s)\) is a (homotopy) surjection and \(t\) is an injection. From this we can deduce that \(s'\) must also be a homotopy surjection. However, this now verifies that computable choice does indeed hold in the reflective subuniverse.
\end{proof}

\section{Relativised computable choice}
\label{sec:relat-churchs-thes}

In this section we will show computable choice is downwards absolute. That is, we will prove that it holds in the reflective subuniverses for oracle modalities, as long as both computable choice and Markov induction hold externally. We will refer to the statement of computable choice in the reflective subuniverse as \emph{relativised computable choice}.

In order to do this, the first step is to show that we can always encode elements of \(\bigcirc_\chi \NN\) as natural numbers. We will use a similar construction to the one used when adding an oracle to a pca \cite[Theorem 1.7.5]{vanoosten}, and also to the synthetic definition of Turing reducibility due to Forster, Kirst and M\"{u}ck \cite{forsterkirstmuck}. The basic idea is that any element of \(\bigcirc_\chi \NN\) is either already an element of \(\NN\) or there is some Turing machine (coded by a natural number) that takes the value of \(\chi(n)\) as input and then returns the code for an earlier constructed element of \(\bigcirc_\chi \NN\). Since this is a coproduct of two countable sets, it is countable. We will eventually write \(\psi\) for the resulting function \(\NN \to \partial \bigcirc_{\chi} \NN\). To get the main results of this section, we will need to have a clear picture of what it means for \(\psi(e)\) to be defined, i.e. to be a total element of \(\bigcirc_{\chi} \NN\). To do this, we will use finite lists of numbers to witness that \(\psi(e)\) is defined. We think of the length of the list as the number of oracle queries needed to complete the computation and each element of the list as the number of steps that the Turing machine takes between each query.
In the below we write finite lists of numbers as \(\vec{k} = [k_1,\ldots,k_n]\) and \(\lst(\NN)\) for the type of all such lists. We fix a bijection \(\theta : \NN \stackrel{\cong}{\to} \NN + (\NN \times \NN)\).

Suppose we are given an oracle \(\chi : \NN \to \nabla_0 \NN\). We define the function \(\psi^\chi : \NN \to \partial \bigcirc_\chi \NN\) as follows. We first define elements \(\psi^\chi_{\vec{k}}(x)\) of \(\bigcirc_\chi (\NN + \{\bot\})\) for each \(x : \NN + (\NN \times \NN)\) and \(\vec{k} : \lst(\NN)\). We define \(\psi^\chi_{\vec{k}}(x)\) by induction on the length \(m\) of the list \(\vec{k}\) as follows. We take \(\psi^\chi_{[]}(x)\) to be \(\inl(n)\) when \(x = \inl(n)\), and otherwise \(\inr(\bot)\). For \(m > 0\), \(\psi^\chi_{[k_1,\ldots,k_m]}(\inl(n))\) is always \(\inr(\bot)\) and we take \(\psi^\chi_{[k_1,\ldots,k_m]}(\inr(n, e))\) to be \(\psi^\chi_{[k_1,\ldots,k_{m-1}]}(\theta(\varphi_e^{k_m}(\chi(n))))\).

% Now for \(e : \NN\) and \(\nu : \bigcirc_{\chi} \NN\), we say \(\psi^\chi(e) \downarrow= \nu\) when we have \(\neg \neg \sum_{\vec{k} : \NN^{<\omega}} \psi^\chi_{\vec{k}}(\theta(e)) = \bigcirc_{\chi}(\inl)(\nu)\).

% to be \(z : \bigcirc_\chi \NN\) such that \(\neg \neg \sum_{\vec{k} : \NN^{<\omega}} \psi^\chi_{\vec{k}}(\theta(e)) = z\) (in
% particular undefined if there is no such \(k\)).

We now want to define the element \(\psi^{\chi}(e)\) of \(\partial \bigcirc_{\chi} \NN\). In order to meet the formal definition of \(\partial \bigcirc_{\chi} \NN\) we need a \(\neg \neg\)-stable proposition for the domain, together with an evaluation function \(\psi(e) {\downarrow} \to \bigcirc_{\chi} \NN\). We say \(\psi^{\chi}(e) {\downarrow}\) when we have \(\bigcirc_{\chi} \sum_{\vec{k} : \lst(\NN)} \sum_{n : \NN} \psi^{\chi}_{\vec{k}}(\theta(e)) = \inl(n)\). The evaluation function is given by applying functoriality of \(\bigcirc_{\chi}\) to the relevant projection map. In order to check this is a valid element of \(\partial \bigcirc_{\chi} \NN\), we still need to check that the domain is \(\neg \neg\)-stable, which we do below, using relativised Markov's principle.

\begin{prop}
  \label{prop:psidomstab}
  \[
    \neg \neg \psi^\chi(\theta(e)) {\downarrow} \;\to\;
    \bigcirc_\chi \psi^\chi(\theta(e)) {\downarrow}
  \]
\end{prop}

\begin{proof}
  Note that we have ensured that for all \(\vec{k} : \lst(\NN)\) we can decide, relative to \(\chi\) which component of \(\NN + \{\bot\}\) an element of \(\psi^\chi_{\vec{k}}(\theta(e))\) belongs to and compute the value if it belongs to the \(\NN\) component, by induction on the length of the list \(\vec{k}\). Hence we can apply \(\markov_{\chi}\), via a bijection \(\NN \cong \NN^{\omega}\) to find \(\vec{k}\) such that \(\psi^\chi_{\vec{k}}(\theta(e))\) belongs to \(\NN\), and thereby derive \(\bigcirc_\chi \psi^\chi(\theta(e)) {\downarrow}\).
\end{proof}

\begin{thm}
  \label{thm:htpyisototeq}
  Suppose computable choice. Then for all \(z : \bigcirc_\chi \NN\) there merely exists \(e\) such that \(\psi^\chi(e) \downarrow= z\).
\end{thm}

\begin{proof}
  We prove this by induction on \(z\) using Rijke et al.'s description of \(\bigcirc_\chi \NN\) as a HIT. Since we are proving a proposition, we only need to deal with the point constructors and can ignore the path constructors.

  If \(z\) is of the form \(\eta(n)\), then we can just take \(e := \theta^{-1}(\inl(n))\).

  Next suppose \(z\) is of the form \(\sup(n, f)\) where \(n : \NN\) and \(f : \chi(n) {\downarrow} \to \bigcirc_\chi \NN\). By induction we may assume that whenever \(\chi(n) {\downarrow}\) there merely exists \(d\) such that \(\psi^\chi(d) \downarrow= f(\chi(n))\). Since \(\psi^\chi(d) \downarrow= f(\chi(n))\) is double negation stable by Proposition~\ref{prop:psidomstab}, we can apply computable choice here to show there exists \(e : \NN\) such that whenever \(\chi(n) \downarrow= m\), we have \(\varphi_{e}(m) {\downarrow}\) and \(\psi^\chi(\varphi_{e}(m)) \downarrow= f(\chi(n))\). It follows that \(\psi^\chi(\inr(n, e)) \downarrow= z\), as required.
\end{proof}

\begin{defn}
  Let \(R : \NN \times \bigcirc_{\chi} \NN \to \Omegann\). We say \(e : \NN\) \emph{computes} \(R\) if for all \(n : \NN\), if \(\neg \neg \sum_{m : \NN} R n m\) holds then \(\psi(\varphi_{e}(n)) \downarrow\) and \(R n \psi(\varphi_{e}(n))\).
\end{defn}

\begin{thm}
  \label{thm:relcompchoice}
  Suppose computable choice. Let \(R : \NN \times \bigcirc_{\chi}\NN \to \Omegann\) be a \(\neg \neg\)-stable relation such that for all \(n : \NN\) we have \(\neg \neg \sum_{\mu : \bigcirc_{\chi} \NN} R n \mu \to \| \sum_{\mu : \bigcirc_{\chi} \NN} R n \mu \|\). Then \(R\) is computed by some \(e : \NN\).
\end{thm}

\begin{proof}
  We define a relation \(R' : \NN \times \NN \to \Omegann\) as follows. For \(n, m : \NN\) we define \(R' n m := \psi(m) {\downarrow} \times R n \psi(m)\). Note that \(R'\) maps into \(\neg \neg\)-stable propositions by Proposition~\ref{prop:psidomstab}.

  Now \(\sum_{m : \NN} R' n m\) implies \(\sum_{\mu : \bigcirc_{\chi} \NN} R n \mu\), since we can just take \(\mu := \psi(m)\). Hence from \(\neg \neg \sum_{m : \NN} R' n m\) we can deduce \(\neg \neg \sum_{\mu : \bigcirc_{\chi} \NN} R n \mu\) and thereby \(\| \sum_{\mu : \bigcirc_{\chi} \NN} R n \mu \|\) by assumption. However, we can then apply Theorem~\ref{thm:htpyisototeq} to deduce that there merely exists \(m : \NN\) such that \(R' n m\).

  We can therefore now apply computable choice to \(R'\) to get the result.
\end{proof}

\begin{cor}
  Suppose that \(\chi' \leq_{T} \chi\) for \(\chi, \chi' : \NN \to \nabla \NN\), then there is \(e : \NN\) such that for all \(n : \NN\), we have \(\chi'(n) = \erase(\psi(\varphi_{e}(n)))\).
\end{cor}

\begin{proof}
  We apply Theorem~\ref{thm:relcompchoice} with \(R n \mu := \chi'(n) = \erase(\mu)\). We show the required hypothesis of \(\prod_{n : \NN} \| \sum_{\mu : \bigcirc_{\chi} \NN} R n \mu \|\) as follows. By the assumption \(\chi' \leq_{T} \chi\) we can show \(\bigcirc_{\chi} \chi'(n) {\downarrow}\) for all \(n : \NN\). Recall that \(\chi'(n)\) is implemented as a \(\Sigma\)-type \(\sum_{m : \NN} \chi'(n) \downarrow= m\). Since \(\bigcirc_{\chi}\) preserves \(\Sigma\)-types and \(R n \mu\) is already \(\neg\neg\)-stable, and thereby \(\bigcirc_{\chi}\)-modal, we in fact have a witness of the untruncated \(\sum_{\mu : \bigcirc_{\chi} \NN} R n \mu\).
\end{proof}

\begin{rmk}
  We can see the above corollary as a partial justification for the assertion that if \(\chi' \leq_{T} \chi\) in our sense, then there is an oracle Turing machine that computes \(\chi'\) from \(\chi\). Strictly speaking, a complete proof would also require checking that the partial function \(\psi \circ \varphi_{e}\) can be implemented as an oracle Turing machine when \(\varphi_{e}\) is the \(e\)th computable function, and that every oracle Turing machine appears as such a function.
\end{rmk}

\begin{cor}
  \label{cor:oraclemodtrunccommute}
  Assuming computable choice, for every \(P : \bigcirc_{\chi} \NN \to \Omegann\) the type \(\| \sum_{\nu : \bigcirc_{\chi}\NN} P n \|\) is \(\bigcirc_{\chi}\)-modal.
\end{cor}

\begin{proof}
  Suppose we are given \(n_{0} : \NN\) and a map \(f : \chi(n_{0}) {\downarrow} \to \| \sum_{\bigcirc_{\chi} \NN} P n \|\). We define \(R : \NN \times \bigcirc_{\chi} \NN \to \Omegann\) by \(R m \nu\) when \(\chi(n_{0}) \downarrow= m\) and \(P \nu\). Applying Theorem~\ref{thm:relcompchoice} tells us that there exists \(e : \NN\) such that if \(\chi_{n_{0}} \downarrow= m\), then \(\psi(\varphi_{e}(m)) \downarrow\) and \(P \psi(\varphi_{e}(m))\). Given such an \(e\), we define \(f' : \chi(n_{0}) {\downarrow} \to \bigcirc_{\chi} \NN\) by taking \(f'(m)\) to be \(\psi(\varphi_{e}(m))\). We can now apply the fact that \(\bigcirc_{\chi} \NN\) is modal to find the diagonal filler, which is an element \(\nu\) of \(\bigcirc_{\chi} \NN\) such that if \(\chi(n_{0}) {\downarrow}\), then \(\nu = \psi(\varphi_{e}(m))\). Since \(P\) is \(\neg\neg\)-stable and \(\chi(n_{0}) {\downarrow}\) is \(\neg \neg\)-dense, this implies that we in fact have \(P \nu\). Hence we have an element of \(\| \sum_{\nu : \bigcirc_{\chi} \NN} P \nu \|\), as required.
\end{proof}

\begin{rmk}
  In general we expect that if we are given an accessible modality generated by projective types then \(-1\)-truncation preserves modal types. In this case, we are not able to show in cubical assemblies that the proposition \(\chi(n) {\downarrow}\) is projective, but using computable choice we have used a related argument to show the special case that \(\| \bigcirc_{\chi} \NN \|\) is \(\bigcirc_{\chi}\)-modal.
\end{rmk}

\begin{cor}
  Suppose that both Markov induction and computable choice are true. Then they both hold in the reflective subuniverse corresponding to any oracle modality. That is, they are ``downwards absolute for oracle modalities.''
\end{cor}

\begin{proof}
  We have already seen that Markov induction holds in Theorem~\ref{thm:markovindinternal}. The remainder of the theorem is a proof that computable choice holds internally. By unfolding the interpretation of the statement in the subuniverse, we can see that this almost follows directly from Theorem~\ref{thm:relcompchoice}. The only issue is that the interpretation of the hypothesis \eqref{eq:compchoicehyp} is \(\prod_{n : \NN} \neg\neg \sum_{m : \NN} R n m \to \bigcirc_{\chi} \| \sum_{m : \NN} R n m \|\), whereas we need a statement of the form \(\prod_{n : \NN} \neg\neg \sum_{m : \NN} R' n m \to \| \sum_{\mu : \bigcirc_{\chi} \NN} R' n \mu \|\) to apply Theorem~\ref{thm:relcompchoice}. However, Corollary~\ref{cor:oraclemodtrunccommute} tells us that \(\| \sum_{\mu : \bigcirc_{\chi} \NN} \prod_{m : \NN}\mu \downarrow= m \to R n m \|\) is \(\bigcirc_{\chi}\)-modal, and so we can eliminate out to get a map \(\bigcirc_{\chi} \| \sum_{m : \NN} R n m \| \to \| \sum_{\mu : \bigcirc_{\chi} \NN} \prod_{m : \NN}\mu \downarrow= m \to R n m \|\). We can then apply Theorem~\ref{thm:relcompchoice} with \(R'n \mu := \mu \downarrow= m \to R n m\) to get the result.
\end{proof}

The above Corollary demonstrates that synthetic proofs using Markov induction and computable choice can be ``relativised to an oracle.'' That is, they can be carried out internally in the reflective subuniverse corresponding to an oracle. This corresponds to the concept of relativisation in computability theory, which is the idea that many proofs about computable functions apply equally well to functions computable relative to an oracle.

Although we needed the full strength of computable choice to show that any form of Church's thesis holds in the reflective subuniverse, once we have it we can derive the other, weaker versions, such as relativised versions of extended Church's thesis and Church's thesis itself:
\begin{cor}[Relativised extended Church's thesis]
  \label{cor:rect}
  For any function \(f : \NN \to \partial \bigcirc_{\chi} \NN\) there exists \(e : \NN\) such that for all \(n\), if \(f(n) {\downarrow}\), then \(\psi(\varphi_{e}(n)) \downarrow= f(n)\).
\end{cor}

\begin{proof}
  We apply relativised computable choice with \(R n \mu := f(n) \downarrow= \mu\).
\end{proof}

\subsection{Turing jump}
\label{sec:turing-jump}

An important corollary of relativised computable choice is that we can now relativise the construction of the halting set and proof that it is non computable (Proposition~\ref{prop:haltingset}). This gives us the following definition of Turing jump.

\begin{defn}
  Given \(\chi : \NN \to \nabla_0 2\), we define the \emph{Turing jump} \(\hat{\chi} : \NN \to \nabla_0 2\) as follows.
  \begin{displaymath}
    \hat{\chi}(n) :=
    \begin{cases}
      1 & \psi(\varphi_n(n)) \downarrow= \eta(0) \\
      0 & \text{otherwise}
    \end{cases}
  \end{displaymath}
\end{defn}

\begin{thm}
  Suppose computable choice and Markov induction. Then \(\hat{\chi} \nleq_T \chi\).
\end{thm}

\begin{proof}
  Suppose that we did have \(\hat{\chi} \leq_T \chi\). This is precisely a witness of \(\prod_{n : \NN} \bigcirc_\chi(\hat{\chi}(n){\downarrow})\), and recall that \(\hat{\chi}(n){\downarrow} := \sum_{m : 2} (\hat{\chi}(n) \downarrow= m) \). By \cite[Lemma 1.24]{rijkeshulmanspitters} we have \(\bigcirc_\chi(\hat{\chi}(n){\downarrow}) \simeq \sum_{z : \bigcirc_\chi 2} \bigcirc_\chi \hat{\chi}(n) \downarrow= m\). We can extract from this a map \(f : \NN \to \bigcirc_\chi 2\) together with a proof that \(\mathtt{erase}(f(n)) = \hat{\chi}(n)\) for all \(n\), i.e. \(\hat{\chi}\) factors through \(\mathtt{erase} : \bigcirc_\chi 2 \to \nabla_0 2\). We then have that \(f\) must be computable by some \(e : \NN\). However, this gives a contradiction, since we have \(\psi(\varphi_e(e)) \downarrow= f(e) \neq \psi(\varphi_e(e))\).
\end{proof}

\subsection{Local continuity}
\label{sec:local-continuity}

We can also use relativised computable choice to give us a ``local continuity'' theorem for our definition of Turing reducibility, verifying that each oracle computation only requires a finite number of queries:
\begin{thm}
  \label{thm:computemodulus}
  Let \(\chi\) be an oracle \(\NN \to \nabla_0 \NN\). Suppose \(\psi^\chi_{[k_1,\ldots,k_n]}(e) {\downarrow}\). Then we can deduce the following.
  \begin{multline*}
    \bigcirc_\chi (\sum_{l \in \lst \NN} \prod_{\chi' : \NN \to \nabla_0 \NN} (\prod_{m \in l} \chi'(m) = \chi(n)) \to \\ \psi^{\chi'}_{[k_1,\ldots,k_n]}(e) {\downarrow} \times \mathtt{erase}(\psi^{\chi'}(e)) = \mathtt{erase}(\psi^\chi(e)))
  \end{multline*}
\end{thm}

\begin{proof}
  We prove this by induction on the length \(m\) of the list \(\vec{k} = [k_1,\ldots,k_m]\).

  For \(\vec{k} = []\), we do not require the oracle, and can just take \(l = []\). We then have \(e = \inl(n)\) and \(\mathtt{erase}(\psi_{\chi'}(e)) = \iota(n) = \mathtt{erase}(\psi_{\chi}(e))\).

  For \(\vec{k} = [k_1,\ldots,k_{m-1},k_{m}]\), we have \(e = \inr(n, e')\) and \(\chi(n) {\downarrow}\) implies \(\varphi_e(\chi(n)) {\downarrow}\) with \(\psi^\chi_{\vec{k}}(e) = \psi^\chi_{[k_1,\ldots,k_{m-1}]}(\theta(\varphi_{e'}^{k_m}(\chi(n))))\). Since we are now working relative to \(\chi\), we can in fact assume \(\chi(n) {\downarrow}\). By the inductive hypothesis, applied with \(e = \theta(\varphi_{e'}^{k_m}(\chi(n)))\), we have, relative to \(\chi\), \(l : \lst(\NN)\) such that \(\prod_{x \in l} \chi'(x) = \chi(x)\) implies \(\mathtt{erase}(\psi^{\chi'}(\theta(\varphi_{e'}^{k_m}(\chi(n))))) = \mathtt{erase}(\psi^\chi(\theta(\varphi_{e'}^{k_m}(\chi(n)))))\). We append \(n\) to \(l\) to obtain a new list \(l'\). If \(\prod_{k \in l'}\chi'(k) = \chi(k)\) then we can calculate as follows.
  \begin{align*}
    \erase(\psi^{\chi'}_{\vec{k}}(e))
    &= \erase(\psi^{\chi'}_{[k_1,\ldots,k_{m-1}]}(\theta(\varphi^{k_m}_{e'}(\chi'(n))))) \\
    &= \erase(\psi^{\chi'}_{[k_1,\ldots,k_{m-1}]}(\theta(\varphi^{k_m}_{e'}(\chi(n))))) \\
    &= \erase(\psi^{\chi}_{[k_1,\ldots,k_{m-1}]}(\theta(\varphi^{k_m}_{e'}(\chi(n)))) \\
    &= \erase(\psi^{\chi}_{\vec{k}}(e)) \qedhere
  \end{align*}
\end{proof}

\begin{cor}
  For any \(\chi : \NN \to \nabla_0 \NN\) and any \(z : \bigcirc_\chi \NN\), there merely exists \(e : \NN\) such that we can compute, relative to \(\chi\), a finite list \(l : \lst (\NN)\) such that whenever \(\chi'(m) = \chi(m)\) for all \(m \in l\) we have \(\psi^{\chi'}(e) {\downarrow}\) with \(\erase(\psi^{\chi'}(e)) = \erase(z)\).
\end{cor}

\begin{proof}
  By Theorem~\ref{thm:htpyisototeq} there merely exists \(e\) such that \(\psi^\chi(e) \downarrow= z\). By Proposition~\ref{prop:psidomstab} we can compute, relative to \(\chi\), the finite list \(\vec{k}\) such that \(\psi^\chi_{\vec{k}}(\theta(e)) = z\). However, we can now apply Theorem~\ref{thm:computemodulus}.
\end{proof}

\begin{rmk}
  The results of this section appear in {\tt RelativisedCC.agda} in the formalisation.
\end{rmk}

\section{Computability and (higher) group theory}
\label{sec:effect-non-comp}

In this section we give an example of the new kind of proof technique that becomes possible by combining the definition of oracle modality in this paper with existing results in HoTT. Our example is a synthetic proof that two Turing degrees are equal as soon as they induce isomorphic permutation groups on \(\NN\). From computability theory, our proof will make use of Lemma~\ref{lem:distinguish}. From HoTT, we will use an approach to group theory due to Buchholtz, Van Doorn and Rijke \cite{buchholtzvdoornrijke}.

As Buchholtz et al. showed there is an equivalence and so equality between the category of groups and the category of pointed, connected, \(1\)-truncated types. Based on this idea, they suggested an approach to group theory where standard constructions in group theory are replaced with arguments that work directly with pointed, connected, \(1\)-truncated types. As they point out, this approach has the advantage of being easier to generalise to higher dimensions. In particular to define \(\infty\)-groups, we simply drop the requirement of \(1\)-truncation. This approach to group theory can help with formalisation, since the group multiplication operation is not extra data that we need to be keep track of, but concatenation of paths, which is already implicit in any type. Furthermore, certain constructions in group theory are more natural from this point of view. Of particular relevance to this proof, the permutation group of a type is just the connected component of that type in the universe, and wreath product can be simply described as a \(\Sigma\)-type \cite[Sections 4.1 and 4.6]{buchholtzvdoornrijke}. We will give a direct proof based on these ideas spelling out the details and with a few further simplifications. We point out that this technique makes essential use of types that are not \(0\)-truncated, and so illustrates the benefits of using cubical assemblies as our semantics rather than the effective topos.

We recall below construction of \emph{homotopy group} of a pointed space, which is used for one direction of the equivalence between the two definitions of group.
\begin{defn}
  Let \(X\) be a type and \(x : X\). The \emph{loop space} at \(x\), \(\Omega(X, x)\) is the identity type \(x =_X x\). The \emph{homotopy group} at \(x\), \(\pi_1(X, x)\) is the \(0\)-truncation \(\| \Omega(X, x) \|_0\) with group multiplication defined as the composition of paths.
\end{defn}

\begin{prop}
  Every pointed map \(f : (X, x) \to (Y, y)\) induces a group homomorphism \(\pi_0(f) : \pi_1(X, x) \to \pi_1(Y, y)\). Every map \(f : X \to Y\) induces a group homomorphism \(\pi_1(X, x) \to \pi_1(Y, f(x))\).
\end{prop}

\begin{thm}
  \label{thm:thtpygp}
  Suppose we are given oracles \(\chi, \chi' : \NN \to \nabla_0 2\) together with a group isomorphism \(\theta : \nabla_0(\Omega(\univ, \bigcirc_\chi \NN)) \simeq \nabla_0(\Omega(\univ, \bigcirc_{\chi'} \NN))\). Then we have \(\neg\neg \chi' \equiv_T \chi\).
\end{thm}

To prove this we will make use of the functions \(F_A\), for \(A : \univ\), and \(D\) defined below.
\begin{displaymath}
  \begin{tikzcd}[row sep = 0em]
    \univ^{A} \ar[r, "F_A"] & \sum_{X : \univ} \univ^X \ar[r, "D"] & \univ \\
    Y : {A} \to \univ \ar[r, |->, "F_A"] & ({A}, Y) & \\
    & (X, Y) \ar[r, |->, "D"] & \sum_{x : X} Y x
  \end{tikzcd}
\end{displaymath}

The first key point is that \(D \circ F_A\) induces an embedding of groups, which we show in two steps.
\begin{lemma}
  \label{lem:fibrelemma}
  For any types \(A, X : \univ\) we have an equivalence \(\fibre_{D \circ F_A}(X) \simeq A^X\)
\end{lemma}

\begin{proof}
  Write \(\univ/A\) for the ``slice category over \(A\)'', i.e. the type \(\sum_{W : \univ} A^W\). Note that we have an equivalence \(e : \univ^A \stackrel{\simeq}{\to} \univ/A\) where \(D \circ F_A = \dom \circ e\). Hence it suffices to show that \(\fibre_{\dom}(X) \simeq A^X\), which we do as follows.
  \begin{align*}
    \fibre_{\dom}(X) &\simeq \sum_{\sum_{W : \univ} A^W} W = X \\
                   &\simeq \sum_{\sum_{W : \univ} W = X} A^W \\
                   &\simeq A^X \qedhere
  \end{align*}
\end{proof}

\begin{lemma}
  \label{lem:df0trunc}
  Suppose that \(A\) and \(X\) are sets.
  The map \(D \circ F_A : \Omega(\univ^{A}, \lambda a.X) \to \Omega(\univ, A \times X)\) defined by action on paths is an embedding.
\end{lemma}

\begin{proof}
  It suffices to show that \(K := \fibre_{D \circ F_A} (A \times X)\) is \(0\)-truncated. However, by Lemma~\ref{lem:fibrelemma}, \(K\) is equivalent to the exponential \(A^{A \times X}\), which is \(0\)-truncated by the assumption that \(A\) and \(X\) are \(0\)-truncated.
\end{proof}

The second key point is that we can view the wreath products of symmetry groups as loop spaces in \(\sum_{X : \univ} \univ^{X}\). In particular we can show by general arguments that concatenation of paths behaves as we would expect for the usual group theoretic definition of wreath product.

More formally, we note that by based path induction for each path \(p : A = B\) and each \(X : A \to \univ\) we can define a path \(\tilde{p} : F_A(X) = F_B(X \circ p^\ast)\). We then have the following lemma.
\begin{lemma}
  \label{lem:sigmaqspecialcase} Suppose we are given \(A, X : \univ\) together with paths \(p : A = A\) and \(q : A \to (X = X)\). Note that \(q \circ p^\ast\) also has type \(A \to X = X\). We then have \(F_A(q \circ p^\ast) = \tilde{p}^{-1} \cdot F_A(q) \cdot \tilde{p}\).
\end{lemma}

\begin{proof}
  This is a special case of the following more general statement that we can prove by based path induction on \(p\). Suppose \(A, B : \univ\), \(X : \univ\), \(p : A = B\) and \(q : A \to X = X\). Then we have the following commutative square in \(\sum_{Z : \univ} \univ^Z\):
  \begin{displaymath}
    \begin{tikzcd}
      (A, \lambda a.X) \ar[r, snake=snake, "\tilde{p}"] \ar[d, snake=snake, "F_A(q)"] & (B, \lambda b.X) \ar[d, snake=snake, "F_B(q \circ p^\ast)"] \\
      (A, \lambda a.X) \ar[r, snake=snake, "\tilde{p}"] & (B, \lambda b.X)
    \end{tikzcd}\qedhere
  \end{displaymath}
\end{proof}

Write \(\ZZ^\ast\) for \(\ZZ + 1\). We are going to apply the above construction with \(A := \bigcirc_{\chi} \ZZ^{\ast}\). We consider the point \(\lambda x.\bigcirc_{\chi} 2\) of \(\univ^{\bigcirc_{\chi} \ZZ^{\ast}}\). In particular, note that the action on paths of \(F_{A}\) and \(D\) gives us group homomorphisms as follows.
\begin{displaymath}
  \Omega(\univ^{\bigcirc_{\chi} \ZZ^{\ast}}, \lambda x.\bigcirc_{\chi} 2) \rightarrow
  \Omega\left(\sum_{{A : \univ}}\univ^{A}, (\bigcirc_{\chi} \ZZ^{\ast}, \lambda x.\bigcirc_{\chi} 2)\right) \rightarrow
  \Omega\left(\univ, \bigcirc_{\chi} \ZZ^{\ast} \times \bigcirc_{\chi} 2\right)
\end{displaymath}
Note also that we have equivalences
\begin{displaymath}
  \bigcirc_{\chi} \ZZ^{\ast} \times \bigcirc_{\chi} 2 \;\cong\;
  \bigcirc_{\chi}\ZZ^{\ast} \times 2 \;\cong\;
  \bigcirc_{\chi} \NN,
\end{displaymath}
which induces an equivalence
\begin{displaymath}
  \Omega\left(\univ, \bigcirc_{\chi} \ZZ^{\ast} \times \bigcirc_{\chi} 2\right) \;\cong\;
  \Omega(\univ, \bigcirc_{\chi} \NN).
\end{displaymath}

We will use the following permutations of \(\ZZ^\ast\) in our proofs. Write \(s : \ZZ^\ast \simeq \ZZ^\ast\) for the permutation of \(\ZZ^\ast\) that shifts by \(1\) on the \(\ZZ\) component and is the identity on the \(1\) component. For \(n : \ZZ\), write \(\tau_n\) for the permutation \(\ZZ^\ast \simeq \ZZ^\ast\) that swaps \(n : \ZZ\) with the unique element of \(1\), and otherwise is the identity.
\begin{lemma}
  \label{lem:swapdecomplemma}
   For all \(n\) we have \(\tau_n = s^n \circ \tau_0 \circ s^{-n}\).
\end{lemma}

\begin{proof}
  It is straightforward to show this directly.
\end{proof}

Finally, note that we can encode \(\chi\) as elements \(k, k' : \Omega(\univ^{\bigcirc_\chi \ZZ^\ast}, \lambda x.\bigcirc_\chi 2)\) as follows. First, given \(n : \NN\) we define \(q : \NN \to \Omega(\univ, \bigcirc_\chi 2)\) internally in \(\bigcirc_\chi\)-modal types as follows, writing \(\mathtt{not}\) for the non trivial path \(2 = 2\) swapping the two elements.
\begin{displaymath}
  q(n) =
  \begin{cases}
    1_2 & \chi(n) = 0 \\
    \mathtt{not} & \chi(n) = 1
  \end{cases}
\end{displaymath}
We define \(k\) and \(k'\) first as elements of \(\ZZ^\ast \to \Omega(\univ, \bigcirc_\chi 2)\), since we can then extend them to maps \(\bigcirc_\chi \ZZ^\ast \to \Omega(\univ, \bigcirc_\chi 2)\) to obtain elements of \(\Omega(\univ^{\bigcirc_\chi \ZZ^\ast}, \lambda x.\bigcirc_\chi 2)\).
\begin{displaymath}
  \begin{gathered}
    k(x) =
    \begin{cases}
      q(| x |) & x \in \ZZ \\
      \mathtt{not} & x \in 1
    \end{cases}
  \end{gathered}
  \qquad
  \begin{gathered}
    k'(x) =
    \begin{cases}
      q(|x|) & x \in \ZZ \\
      1_{2} & x \in 1
    \end{cases}
  \end{gathered}
\end{displaymath}

Note that we have defined \(k, k', \tau_n\) to ensure that the following lemma holds.
\begin{lemma}
  \label{lem:keqtochi}
  For all \(n\), we have \(\chi(n) \downarrow= 0\) iff \(k \neq k \circ \tau_n\), and \(\chi(n) \downarrow= 1\) iff \(k' \neq k' \circ \tau_n\).
\end{lemma}

Finally this gives us the key lemma that tells us how to compute the Turing reduction from a finite list of values of \(\theta\). For convenience, we will assume that \(\theta\) has been composed with appropriate isomorphisms to make it an isomorphism \(\Omega(\univ, \bigcirc_\chi (\ZZ^\ast \times 2)) \to \Omega(\univ, \bigcirc_{\chi'} (\ZZ^\ast \times 2))\).
\begin{lemma}
  \label{lem:htypfromfinval}
  Suppose that \(\theta(D(F_{\ZZ^\ast}(k))) {\downarrow}\), \(\theta(D(F_{\ZZ^\ast}(k'))) {\downarrow}\), \(\theta(D(\tilde{\tau}_0)) {\downarrow}\), and \(\theta(D(\tilde{s})) {\downarrow}\). Then \(\chi \leq_T \chi'\).
\end{lemma}

\begin{proof}
  By Lemma~\ref{lem:swapdecomplemma} we can show that \(\theta(D(\tilde{\tau}_n)) {\downarrow}\) for all \(n\). But by Lemma~\ref{lem:sigmaqspecialcase} we know \(F_{\bigcirc_\chi \ZZ^\ast}(k \circ \tau_n^\ast) = \tilde{\tau}_n^{-1} \cdot F_{\bigcirc_\chi \ZZ^\ast}(k) \cdot \tilde{\tau}_n\), and so also \(\theta(D(F_{\bigcirc_\chi \ZZ^\ast}(k \circ \tau_n^\ast))) {\downarrow}\) and similarly \(\theta(D(F_{\bigcirc_\chi \ZZ^\ast}(k' \circ \tau_n^\ast))) {\downarrow}\). Since \(\bigcirc_\chi\) is lex, it preserves hlevel \cite[Corollary 3.9]{rijkeshulmanspitters}, and so we can apply Lemma~\ref{lem:df0trunc} to see that \(D \circ F_{\bigcirc_\chi \ZZ^\ast}\) is an embedding of groups.

  Therefore by Lemma~\ref{lem:keqtochi} and  we can see that \(\chi(n) \downarrow= 0\) iff \(\theta(D(F_{\bigcirc_\chi \ZZ^\ast}(k))) \neq \theta(D(F_{\bigcirc_\chi \ZZ^\ast}(k \circ \tau_n^\ast)))\) and \(\chi(n) \downarrow= 1\) iff \(\theta(D(F_{\bigcirc_\chi \ZZ^\ast}(k'))) \neq \theta(D(F_{\bigcirc_\chi \ZZ^\ast}(k' \circ \tau_n^\ast)))\). However, we can view the elements of \(\Omega(\univ, \bigcirc_{\chi'} (\ZZ^\ast \times 2))\) as functions computable in \(\chi'\), and so apply Lemma~\ref{lem:distinguish} to deduce \(\bigcirc_{\chi'} (\theta(D(F_{\bigcirc_\chi \ZZ^\ast}(k))) \neq \theta(D(F_{\bigcirc_\chi \ZZ^\ast}(k \circ \tau_n^\ast))) + \theta(D(F_{\bigcirc_\chi \ZZ^\ast}(k'))) \neq \theta(D(F_{\bigcirc_\chi \ZZ^\ast}(k' \circ \tau_n^\ast))))\) and thereby \(\bigcirc_{\chi'} \chi(n) {\downarrow}\).
\end{proof}

\begin{proof}[Proof of Theorem \ref{thm:htpyisototeq}]
  Clearly it suffices to show \(\neg \neg \chi \leq_T \chi'\), since we can then also show \(\neg \neg \chi' \leq_T \chi\) by symmetry and thereby deduce \(\neg \neg \chi' \equiv_T \chi\).

  Note that from Lemma~\ref{lem:htypfromfinval} we can deduce \(\neg \neg \chi' \equiv_T \chi\) from the double negation of the product of the finite list of assumptions \(\theta(D(F_{\ZZ^\ast}(k))) {\downarrow}\), \(\theta(D(F_{\ZZ^\ast}(k'))) {\downarrow}\), \(\theta(D(\tilde{\tau}_0)) {\downarrow}\), and \(\theta(D(\tilde{s})) {\downarrow}\). However, since the list is finite, the double negation of its product is the same as the product of the double negations of each of its elements. However, we can prove the double negation of each element of the list, and so we are done.
\end{proof}

\begin{rmk}
  A proof of the key Lemma~\ref{lem:htypfromfinval} appears in {\tt PermutationGroups.agda} in the formalisation.
\end{rmk}

\section{Some non-topological modalities}
\label{sec:some-observ-modal}

In this section we will give some examples of modalities based on oracle modalities, but combined with higher types. For this, we will use a general construction due to Christensen, Opie, Rijke and Scoccola \cite{christensenopierijkescoccola}.

\begin{defn}[Christensen-Opie-Rijke-Scoccola]
  Let \(\bigcirc\) be a modality. The \emph{suspension} of \(\bigcirc\) is the (necessarily unique) modality \(\Sigma \bigcirc\) such that a type \(A\) is \(\Sigma \bigcirc\)-modal if and only if it is \(\bigcirc\)-separated.

  We write \(\bigcirc^{(n)}\) for the result of iterating the suspension \(k\) times. That is, \(\bigcirc^{(0)} = \bigcirc\) and \(\bigcirc^{(n + 1)} = \Sigma (\bigcirc^{(n)})\).
\end{defn}

Although Christensen et. al. showed that suspension modalities exist in general, we will only need the following easier special case.
\begin{prop}
  Suppose that \(\bigcirc\) is the nullification of a family of types \(B : A \to \univ\). Then \(\Sigma \bigcirc\) is the nullification of the family \(\Sigma B : A \to \univ\), defined by mapping \(a : A\) to \(\Sigma (B(a))\), the suspension of \(B(a)\).
\end{prop}

Although this is an entirely general definition, we can view it as a natural construction in computability theory as follows. We can think of \(\bigcirc_{\chi} X\) as the result of adding new points to \(X\). Namely, we think of \(X\) as only containing computable elements, and \(\bigcirc_{\chi} X\) as also containing new points that can be computed using the oracle \(\chi\). On the other hand, in \(\bigcirc^{(1)}_{\chi} X\) we can use the oracle only to compute new paths, while the points are left unaffected. More generally, in \(\bigcirc^{(n)}_{\chi}\) we can use the oracle \(\chi\) to compute new \(m\)-cells for \(m \geq n\), leaving lower dimensions the same.

We can make this precise through the following arguments. We first consider the easier case of showing that lower dimensions are unaffected. We first observe that the case where \(n > m + 1\) holds in general for purely formal reasons, using the following lemma.
\begin{lemma}
  \label{lem:suspmodalitytoconn}
  For any modality \(\bigcirc\) and any type \(X\), the unit map \(\eta^{(n)} : X \to \bigcirc^{(n)} X\) is \(n - 2\)-connected.
\end{lemma}

\begin{proof}
  We have \(\bigcirc \leq_{T} \| - \|_{-2}\). Hence, \(\bigcirc^{(n)} \leq_{T} \| - \|_{-2}^{(n)} = \| - \|_{n - 2}\). Hence, \(\eta^{(n)}\) is \(\| - \|_{n - 2}\)-connected.
\end{proof}

\begin{prop}
  \label{prop:deletemodminus2}
  If \(n - k \geq 2\) then \(\pi_{k}(\bigcirc^{(n)} X) = \pi_{k}(X)\) for all \(X\).
\end{prop}

\begin{proof}
  By Lemma~\ref{lem:suspmodalitytoconn} the unit map \(Y \to \bigcirc^{(n - k)} Y\) is \(0\)-connected for any \(Y\). The truncation map \(\bigcirc^{(n - k)} Y \to \| \bigcirc^{(n)} Y \|_{0}\) is also \(0\)-connected. Hence the composition \(Y \to \| \bigcirc^{(n - k)} Y \|_{0}\) is \(0\)-connected. Hence \(\| \bigcirc^{(n - k)} Y \|_{0} = \| Y \|_{0}\) by uniqueness of factorisations. Applying this with \(Y = \Omega^{k} X\), we can reason as follows.
  \begin{align*}
    \pi_{k}(\bigcirc^{(n)} X) &= \| \Omega^{k} (\bigcirc^{(n)} X) \|_{0} \\
                              &= \| \bigcirc^{(n - k)} \Omega^{k} X \|_{0} \\
                              &= \| \Omega^{k} X \|_{0} \\
    &= \pi_{k}(X) \qedhere
  \end{align*}
\end{proof}

The case \(n = k + 1\) is harder, and requires the additional assumption that we already know \(\pi_{k}(A)\) is \(\neg \neg\)-separated:
\begin{lemma}
  \label{lem:deletesettrunc}
  Suppose that \(\| X \|_{0}\) is \(\nabla\)-separated. Then for any modality \(\bigcirc\) with \(\bigcirc \leq_{T} \nabla\), we have \(\| \bigcirc^{(1)} X \|_{0} = \| X \|_{0}\).
\end{lemma}

\begin{proof}
  We have in any case a canonical map \(\| X \|_{0} \to \| \bigcirc^{(1)} X \|_{0}\). The map is surjective, i.e. \(-1\)-connected, by the same argument as in Proposition~\ref{prop:deletemodminus2}. Hence, we just need to check that the map is injective. Since truncation is surjective, it suffices to show that for \(x, y : X\), if \(| \eta^{\bigcirc^{(1)}}(x) |_{0} = | \eta^{\bigcirc^{(1)}}(y) |_{0}\), then \(| x |_{0} = | y |_{0}\). However, we can calculate
  \begin{align*}
    | \eta^{\bigcirc^{(1)}}(x) |_{0} = | \eta^{\bigcirc^{(1)}}(y) |_{0}
    &= \| \eta^{\bigcirc^{(1)}}(x) = \eta^{\bigcirc^{(1)}}(y) \|_{-1} \\
    &= \| \bigcirc (x = y) \|_{-1} \\
    &\leq \| \nabla(x = y) \|_{-1} \\
    &\leq \| \nabla(|x|_{0} = |y|_{0}) \|_{-1} \\
    &= \| |x|_{0} = |y|_{0} \|_{-1} \\
    &= |x|_{0} = |y|_{0} \qedhere
  \end{align*}
\end{proof}

\begin{prop}
  \label{prop:deletemodminus1}
  Suppose that \(X\) is a type such that \(\pi_{n}(X)\) is \(\neg\neg\)-separated for some \(n \geq 1\). Then \(\pi_{n}(\bigcirc^{(n + 1)} X) = \pi_{n}(X)\).
\end{prop}

\begin{proof}
  We first observe the following equality.
  \begin{align*}
    \pi_{n}(\bigcirc^{(n + 1)} X) &= \| \Omega^{n}(\bigcirc^{(n + 1)} X) \|_{0} \\
    &= \| \bigcirc^{(1)} \Omega^{n} X \|_{0}
  \end{align*}
  However, we can now apply Lemma~\ref{lem:deletesettrunc} to \(\Omega^{n} X\) to verify the result.
\end{proof}

\begin{example}
  Recall that \(\pi_{n}(\mathbb{S}^{n}) = \ZZ\), e.g. from \cite[Theorem 8.6.17]{hottbook}), which is \(\neg \neg\)-separated. Hence for \(k > n\), we have \(\pi_{n}(\bigcirc^{(k)}_{\chi} \mathbb{S}^{n}) = \ZZ\).
\end{example}

Finally, we give some examples to illustrate that \(\bigcirc^{(n)} X\) can affect the homotopy groups of sufficiently high dimension.
\begin{example}
  To illustrate that the condition of \(\neg \neg\)-separation is necessary in Proposition~\ref{prop:deletemodminus1}, we give the following class of examples. Let \(\chi : \NN \to \nabla_{0} \NN\) be non computable, and fix \(k : \NN\). Take \(Y_{n}\) to be the result of taking the suspension of \(\chi(n) {\downarrow}\) \(k + 1\) times, \(Y_{n} := \Sigma^{k + 1} \chi(n) {\downarrow}\). Note that each \(Y_{n}\) is a generator of \(\bigcirc_{\chi}^{(k + 1)}\) and so \(\bigcirc_{\chi}^{(k + 1)}Y_{n} = 1\), and so also \(\pi_{k} (\bigcirc_{\chi}^{(k + 1)}Y_{n}) = 1\). We can compute \(\pi_{k}(Y_{n})\) as follows. We in fact get an explicit description of \(\Omega^{k}(Y_{n})\). Recall that taking the suspension of a type is the same as taking the join with \(2\) \cite[Lemma 8.5.10]{hottbook}. Hence we can calculate as follows.
  \begin{align*}
    Y_{n} &= \Sigma^{k + 1} \chi(n) {\downarrow} \\
          &= \underbrace{2 \ast \ldots \ast 2}_{k + 1 \text{ times}} \ast \chi(n) {\downarrow} \\
    &= \mathbb{S}^{k} \ast \chi(n) {\downarrow}
  \end{align*}
  However, \(\chi(n){\downarrow}\) is a proposition, and so \(- \ast \chi(n){\downarrow}\) is a closed modality, which is lex \cite[Example 3.14]{rijkeshulmanspitters}. We deduce that \(\Omega^{k}(\mathbb{S}^{k} \ast \chi(n) {\downarrow}) = \Omega^{k}(\mathbb{S}^{k}) \ast \chi(n) {\downarrow}\), and so \(\pi_{k} (Y_{n}) = \| \Omega^{k}(\mathbb{S}^{k}) \ast \chi(n) {\downarrow} \|_{0}\). Note that \(\| \Omega^{k}(\mathbb{S}^{k}) \ast \chi(n) {\downarrow} \|_{0}\) is modal for the closed modality \(- \ast \chi(n){\downarrow}\) since if \(\chi(n) {\downarrow}\) then it is contractible, and furthermore that \(\| \Omega^{k}(\mathbb{S}^{k}) \|_{0} \ast \chi(n) {\downarrow}\) is \(0\)-truncated, again using that \(- \ast \chi(n){\downarrow}\) is lex. It follows that in fact \(\| \Omega^{k}(\mathbb{S}^{k}) \ast \chi(n) {\downarrow} \|_{0} = \| \Omega^{k}(\mathbb{S}^{k}) \|_{0} \ast \chi(n) {\downarrow}\). We recall that \(\pi_{k} \mathbb{S}^{k} = \mathbb{Z}\), and so \(\pi_{k}(Y_{n}) = \ZZ \ast \chi(n){\downarrow}\). Hence from \(\pi_{k}(\bigcirc_{\chi}^{(k + 1)} Y_{n}) = \pi_{k} Y_{n}\), we could deduce \(\chi(n) {\downarrow}\). Therefore, when \(\chi\) is non computable we can show, assuming Church's thesis, that it is false that for all \(n\), \(\pi_{k}(\bigcirc_{\chi}^{(k + 1)}Y_{n}) = \pi_{k}(Y_{n})\).
\end{example}

\begin{example}
  The easiest example where the first homotopy group \(\pi_{1}(X)\) is \(\neg \neg\)-separated and altered by \(\bigcirc^{(1)}\) is as follows. Recall that \(\Omega(\sone) = \ZZ\). Hence as a special case of \cite[Corollary 3.5]{christensenopierijkescoccola}, we have \(\Omega(\bigcirc^{(1)}_{\chi}(\sone)) = \bigcirc_{\chi} \ZZ\), and so \(\pi_{1}(\bigcirc^{(1)}_{\chi}(\sone)) = \bigcirc_{\chi} \ZZ \neq \ZZ = \pi_{1}(\sone)\), with the inequality holding when \(\bigcirc_{\chi}\) is non trivial, e.g. when \(\chi\) is non computable and Church's thesis holds.
\end{example}

\section{Conclusion and future work}
\label{sec:conclusion}

We have given a new synthetic definition of Turing reducibility using modalities. We formulated the axiom of Markov induction, and demonstrated its use in some synthetic proofs of Turing reducibility. Although simple, each case required an unbounded search that would not otherwise be possible. Moreover, we observe that these examples also made use of an unbounded, albeit finite, number of oracle queries. As such we can see that the arguments in these examples would not apply to many-one reducibility or truth table reducibility, therefore requiring the full power of Turing reducibility. We also gave an explicit proof that this definition of Turing reducibility is strictly stronger than weak truth table reducibility in Section~\ref{sec:comparison-with-weak}. We have also seen some first signs of interaction between computability theory and homotopy theory in Sections \ref{sec:effect-non-comp} and \ref{sec:some-observ-modal}.

However, this is intended as laying the foundation for future work. We suggest a few possible directions for further research below.
\subsection*{Synthetic computability theory} We have developed an approach to synthetic Turing reducibility based on modalities in cubical assemblies. This was based on Hyland's embedding of Turing degrees in the lattice of subtoposes of the effective topos but differs in a few key ways. The main difference is that replacing Lawvere-Tierney topologies with topological modalities gives us the potential to combine computability theory with higher types, with some promising signs of such interaction in Sections \ref{sec:effect-non-comp} and \ref{sec:some-observ-modal}. However, even putting aside higher types, we also developed an approach that is more synthetic than Hyland's. Instead of working externally to cubical assemblies, we identified the axioms of Markov induction and computable choice as sufficient for carrying out some simple examples of synthetic arguments, which we have formalised using the cubical Agda proof assistant. In particular the axiom of Markov's principle, commonly assumed in recursive constructive mathematics, is apparently insufficient for analogous arguments using computations with an oracle. In contrast, the stronger axiom of Markov induction has allowed us to show, for example, that we obtain the same class of modalities whether we only use functions \(\NN \to \nabla 2\) as oracles, or whether we allow all functions \(\NN \to \nabla \NN\).

\subsection*{Separating Turing degrees with homotopy theory} From Theorem~\ref{thm:thtpygp} we can see that, at least in principle, it is possible to study when two Turing degrees are equal by studying homotopy groups, since oracles \(\chi, \chi' : \NN \to \nabla_{0} 2\) have the same Turing degree precisely when the homotopy groups \(\pi_{1}(\nabla \univ, \eta (\bigcirc_{\chi}\NN))\) and \(\pi_{1}(\nabla \univ, \eta(\bigcirc_{\chi'} \NN))\) are isomorphic. We leave it for future work to determine if this is useful in practice for computability theory. However, the idea seems promising: showing that two Turing degrees are separate is a common task in computability theory, whereas in homotopy theory it is common to show two spaces are different by analysing their homotopy or (co)homology groups.

\subsection*{Higher dimensional computability theory} One of the main motivations for using this formulation based on higher modalities is to make it easy to generalise computability theory to higher dimensions in order to find new interactions between computability theory and homotopy theory. We propose one such generalisation here, in a similar spirit to Kihara's generalisations of Turing reducibility based on Lawvere-Tierney topologies in the effective topos \cite{kiharaltt, kiharasdst}, which will include the examples we saw in Section~\ref{sec:some-observ-modal}.

Recall that we recovered the original definition of Turing degree by considering oracles of the form \(\NN \to \nabla_0 2\), which leads to us nullifying only \(\neg\neg\)-dense propositions, or in other words propositions that turn out to be simply true when we add classical logic. This is natural in computability theory: we don't want to study the structure of the function itself, but only what we can compute with it. We therefore propose generalising the definition of Turing degree to include the nullification of countable families of \(\nabla\)-connected types, i.e. spaces that become contractible when we erase all the computational information leaving only an ordinary \(\infty\)-groupoid. This eliminates any structure that can already be studied purely using existing techniques in homotopy theory, leaving only new structure that can only be seen through the interaction of computation and homotopy theory.

We observe that suspension preserves \(\nabla\)-connected types, and so this does indeed include the suspension modalities of Section~\ref{sec:some-observ-modal}.

\subsection*{Automatic relativisation of proofs} As remarked in the introduction, an advantage to working with lex modalities rather than more general monads is that they correspond to reflective subuniverses, which are complete models of HoTT in themselves. This allows us to carry out any proof in HoTT internally in the reflective subuniverse to obtain a relativised version. Although this was important in formulating the theorems and definitions, and giving the intuitive ideas for the proofs, it is not used directly in the computer formalisation, which instead uses syntactic sugar, following the standard approach for working with monads in functional programming. However, this was almost entirely due to the practical difficulties involved in translating proofs by hand. Future work in this area would greatly benefit from the development of an automated approach, using Agda reflection for example.

\bibliographystyle{alpha}
\bibliography{mybib}

\newcommand{\etalchar}[1]{$^{#1}$}
\begin{thebibliography}{CCHM18}

\bibitem[ABC{\etalchar{+}}21]{abchfl}
Carlo Angiuli, Guillaume Brunerie, Thierry Coquand, Robert Harper, Kuen-Bang
  Hou~(Favonia), and Daniel~R. Licata.
\newblock Syntax and models of cartesian cubical type theory.
\newblock {\em Mathematical Structures in Computer Science}, 31(4):424–468,
  2021.

\bibitem[Awo19]{awodey19}
Steve Awodey.
\newblock {A Quillen model structure on the category of cartesian cubical
  sets}.
\newblock Preprint available at
  \url{https://github.com/awodey/math/blob/master/QMS/qms.pdf}, 2019.

\bibitem[Bau06]{bauerfirststeps}
Andrej Bauer.
\newblock First steps in synthetic computability theory.
\newblock {\em Electronic Notes in Theoretical Computer Science}, 155:5--31,
  2006.
\newblock Proceedings of the 21st Annual Conference on Mathematical Foundations
  of Programming Semantics (MFPS XXI).

\bibitem[Bau17]{bauerfixedpoint}
Andrej Bauer.
\newblock {On fixed-point theorems in synthetic computability}.
\newblock {\em Tbilisi Mathematical Journal}, 10(3):167 -- 181, 2017.

\bibitem[Bau20]{bauerexcursion}
Andrej Bauer.
\newblock Synthetic mathematics with an excursion into computability theory
  (slide set), 2020.
\newblock University of Wisconsin Logic seminar.

\bibitem[BCH14]{bchcubicalsets}
Marc Bezem, Thierry Coquand, and Simon Huber.
\newblock {A Model of Type Theory in Cubical Sets}.
\newblock In Ralph Matthes and Aleksy Schubert, editors, {\em 19th
  International Conference on Types for Proofs and Programs (TYPES 2013)},
  volume~26 of {\em Leibniz International Proceedings in Informatics (LIPIcs)},
  pages 107--128, Dagstuhl, Germany, 2014. Schloss Dagstuhl--Leibniz-Zentrum
  fuer Informatik.

\bibitem[BCH18]{bchunivalence}
Marc Bezem, Thierry Coquand, and Simon Huber.
\newblock The univalence axiom in cubical sets.
\newblock {\em Journal of Automated Reasoning}, Jun 2018.

\bibitem[Bee85]{beeson85}
Michael Beeson.
\newblock {\em Foundations of Constructive Mathematics: Metamathematical
  Studies}.
\newblock Springer, 1985.

\bibitem[BvDR18]{buchholtzvdoornrijke}
Ulrik Buchholtz, Floris van Doorn, and Egbert Rijke.
\newblock Higher groups in homotopy type theory.
\newblock In {\em Proceedings of the 33rd Annual ACM/IEEE Symposium on Logic in
  Computer Science}, LICS '18, page 205–214, New York, NY, USA, 2018.
  Association for Computing Machinery.

\bibitem[CCHM18]{coquandcubicaltt}
Cyril Cohen, Thierry Coquand, Simon Huber, and Anders M{\"o}rtberg.
\newblock {Cubical Type Theory: A Constructive Interpretation of the Univalence
  Axiom}.
\newblock In Tarmo Uustalu, editor, {\em 21st International Conference on Types
  for Proofs and Programs (TYPES 2015)}, volume~69 of {\em Leibniz
  International Proceedings in Informatics (LIPIcs)}, pages 5:1--5:34,
  Dagstuhl, Germany, 2018. Schloss Dagstuhl--Leibniz-Zentrum fuer Informatik.

\bibitem[CORS20]{christensenopierijkescoccola}
J.~Daniel Christensen, Morgan Opie, Egbert Rijke, and Luis Scoccola.
\newblock Localization in homotopy type theory.
\newblock {\em Higher Structures}, 4:1--32, 2020.

\bibitem[EK17]{escardoknapp}
Mart{\'i}n~H. Escard{\'o} and Cory~M. Knapp.
\newblock {Partial Elements and Recursion via Dominances in Univalent Type
  Theory}.
\newblock In Valentin Goranko and Mads Dam, editors, {\em 26th EACSL Annual
  Conference on Computer Science Logic (CSL 2017)}, volume~82 of {\em Leibniz
  International Proceedings in Informatics (LIPIcs)}, pages 21:1--21:16,
  Dagstuhl, Germany, 2017. Schloss Dagstuhl--Leibniz-Zentrum fuer Informatik.

\bibitem[Esc13]{escardoeffectfulforcing}
Martín Escardó.
\newblock Continuity of {G}ödel's system {T} definable functionals via
  effectful forcing.
\newblock {\em Electronic Notes in Theoretical Computer Science}, 298:119--141,
  2013.
\newblock Proceedings of the Twenty-ninth Conference on the Mathematical
  Foundations of Programming Semantics, MFPS XXIX.

\bibitem[FAS21]{finsteralliouxsozeau}
Eric Finster, Antoine Allioux, and Matthieu Sozeau.
\newblock Types are internal \(\infty\)-groupoids.
\newblock In {\em 2021 36th Annual ACM/IEEE Symposium on Logic in Computer
  Science (LICS)}, pages 1--13, 2021.

\bibitem[FJ23]{forsterjahn}
Yannick Forster and Felix Jahn.
\newblock {Constructive and Synthetic Reducibility Degrees: Post's Problem for
  Many-one and Truth-table Reducibility in Coq}.
\newblock In {\em {CSL 2023 - 31st EACSL Annual Conference on Computer Science
  Logic}}, Warsaw, Poland, February 2023.

\bibitem[FKM23]{forsterkirstmuck}
Yannick Forster, Dominik Kirst, and Niklas M{\"u}ck.
\newblock Oracle computability and {T}uring reducibility in the calculus of
  inductive constructions.
\newblock In Chung-Kil Hur, editor, {\em Programming Languages and Systems},
  pages 155--181, Singapore, 2023. Springer Nature Singapore.

\bibitem[For21]{forsterthesis}
Yannick Forster.
\newblock {\em Computability in Constructive Type Theory}.
\newblock PhD thesis, Saarland University, 2021.

\bibitem[GH04]{gambinohylanddepw}
Nicola Gambino and Martin Hyland.
\newblock Wellfounded trees and dependent polynomial functors.
\newblock In Stefano Berardi, Mario Coppo, and Ferruccio Damiani, editors, {\em
  Types for Proofs and Programs: International Workshop, TYPES 2003, Torino,
  Italy, April 30 - May 4, 2003, Revised Selected Papers}, pages 210--225.
  Springer Berlin Heidelberg, Berlin, Heidelberg, 2004.

\bibitem[HS00]{hancocksetzer}
Peter Hancock and Anton Setzer.
\newblock Interactive programs in dependent type theory.
\newblock In Peter~G. Clote and Helmut Schwichtenberg, editors, {\em Computer
  Science Logic}, pages 317--331, Berlin, Heidelberg, 2000. Springer Berlin
  Heidelberg.

\bibitem[Hub18]{hubercanonicity}
Simon Huber.
\newblock Canonicity for cubical type theory.
\newblock {\em Journal of Automated Reasoning}, Jun 2018.

\bibitem[HvOS06]{hofmannvoostenstreicher}
M.~Hofmann, J.~van Oosten, and T.~Streicher.
\newblock Well-foundedness in realizability.
\newblock {\em Archive for Mathematical Logic}, 45(7):795--805, Oct 2006.

\bibitem[Hyl82]{hylandeff}
J.~M.~E. Hyland.
\newblock The effective topos.
\newblock In A.~S. Troelstra and D.~van Dalen, editors, {\em The L. E. J.
  Brouwer Centenary Symposium Proceedings of the Conference held in
  Noordwijkerhout}, volume 110 of {\em Studies in Logic and the Foundations of
  Mathematics}, pages 165 -- 216. Elsevier, 1982.

\bibitem[Joh02]{theelephant}
Peter~T. Johnstone.
\newblock {\em Sketches of an Elephant: A Topos Theory Compendium}.
\newblock Oxford logic guides. Oxford University Press, 2002.

\bibitem[Kih22]{kiharasdst}
Takayuki Kihara.
\newblock Rethinking the notion of oracle: A link between synthetic descriptive
  set theory and effective topos theory.
\newblock arXiv preprint, arXiv:2202.00188, 2022.

\bibitem[Kih23]{kiharaltt}
Takayuki Kihara.
\newblock Lawvere-{T}ierney topologies for computability theorists.
\newblock {\em Transactions of the American Mathematical Society. Series B},
  10:48--85, 2023.

\bibitem[Kle59]{kleenerfqft1}
S.~C. Kleene.
\newblock Recursive functionals and quantifiers of finite types {I}.
\newblock {\em Transactions of the American Mathematical Society}, 91(1):1--52,
  1959.

\bibitem[LOPS18]{lops}
Daniel~R. Licata, Ian Orton, Andrew~M. Pitts, and Bas Spitters.
\newblock {Internal Universes in Models of Homotopy Type Theory}.
\newblock In H{\'e}l{\`e}ne Kirchner, editor, {\em 3rd International Conference
  on Formal Structures for Computation and Deduction (FSCD 2018)}, volume 108
  of {\em Leibniz International Proceedings in Informatics (LIPIcs)}, pages
  22:1--22:17, Dagstuhl, Germany, 2018. Schloss Dagstuhl--Leibniz-Zentrum fuer
  Informatik.

\bibitem[MLM94]{moerdijkmaclane}
Saunders Mac~Lane and Ieke Moerdijk.
\newblock {\em Sheaves in Geometry and Logic: A First Introduction to Topos
  Theory}.
\newblock Universitext. Springer New York, 1994.

\bibitem[Mog91]{moggi91}
Eugenio Moggi.
\newblock Notions of computation and monads.
\newblock {\em Information and Computation}, 93(1):55--92, 1991.
\newblock Selections from 1989 IEEE Symposium on Logic in Computer Science.

\bibitem[OP16]{pittsortoncubtopos}
I.~Orton and A.~M. Pitts.
\newblock Axioms for modelling cubical type theory in a topos.
\newblock In J.-M. Talbot and L.~Regnier, editors, {\em 25th EACSL Annual
  Conference on Computer Science Logic ({CSL} 2016)}, volume~62 of {\em Leibniz
  International Proceedings in Informatics (LIPIcs)}, pages 24:1--24:19,
  Dagstuhl, Germany, 2016. Schloss Dagstuhl--Leibniz-Zentrum f\"ur Informatik.

\bibitem[Pho89]{phoarelcompeff}
Wesley Phoa.
\newblock Relative computability in the effective topos.
\newblock {\em Mathematical Proceedings of the Cambridge Philosophical
  Society}, 106(3):419–422, 1989.

\bibitem[Qui16]{quirinthesis}
Kevin Quirin.
\newblock {\em {Lawvere-Tierney sheafification in Homotopy Type Theory}}.
\newblock Theses, {Ecole des Mines de Nantes}, December 2016.

\bibitem[Ric83]{richmanwithouttears}
Fred Richman.
\newblock Church's thesis without tears.
\newblock {\em Journal of Symbolic Logic}, 48(3):797–803, 1983.

\bibitem[Rij23]{rijkebook}
Egbert Rijke.
\newblock {\em Introduction to Homotopy Type Theory}.
\newblock Cambridge Studies in Advanced Mathematics. Cambridge University
  Press, 2023.

\bibitem[Rog87]{rogers}
Hartley Rogers.
\newblock {\em Theory of Recursive Functions and Effective Computability}.
\newblock MIT Press, 1987.

\bibitem[Ros86]{rosolinithesis}
Giuseppe Rosolini.
\newblock {\em Continuity and effectiveness in topoi}.
\newblock PhD thesis, Oxford, 1986.

\bibitem[RSS20]{rijkeshulmanspitters}
Egbert Rijke, Michael Shulman, and Bas Spitters.
\newblock {Modalities in homotopy type theory}.
\newblock {\em {Logical Methods in Computer Science}}, {Volume 16, Issue 1},
  January 2020.

\bibitem[SA21]{sterlingangiuli}
Jonathan Sterling and Carlo Angiuli.
\newblock Normalization for cubical type theory.
\newblock In {\em 2021 36th Annual ACM/IEEE Symposium on Logic in Computer
  Science (LICS)}, pages 1--15, 2021.

\bibitem[SU21]{uemuraswan}
Andrew~W. Swan and Taichi Uemura.
\newblock On {C}hurch's thesis in cubical assemblies.
\newblock {\em Mathematical Structures in Computer Science},
  31(10):1185–1204, 2021.

\bibitem[Swa22]{swandnsprop}
Andrew~Wakelin Swan.
\newblock Double negation stable h-propositions in cubical sets.
\newblock arXiv:2209.15035, 2022.

\bibitem[Uem19]{uemuracubasm}
Taichi Uemura.
\newblock {Cubical Assemblies, a Univalent and Impredicative Universe and a
  Failure of Propositional Resizing}.
\newblock In Peter Dybjer, Jos{\'e}~Esp{\'\i}rito Santo, and Lu{\'\i}s Pinto,
  editors, {\em 24th International Conference on Types for Proofs and Programs
  (TYPES 2018)}, volume 130 of {\em Leibniz International Proceedings in
  Informatics (LIPIcs)}, pages 7:1--7:20, Dagstuhl, Germany, 2019. Schloss
  Dagstuhl--Leibniz-Zentrum fuer Informatik.

\bibitem[{Uni}13]{hottbook}
{Univalent Foundations Program}.
\newblock {\em Homotopy Type Theory: Univalent Foundations of Mathematics}.
\newblock \url{http://homotopytypetheory.org/book}, Institute for Advanced
  Study, 2013.

\bibitem[VMA19]{cubicalagda}
Andrea Vezzosi, Anders M\"{o}rtberg, and Andreas Abel.
\newblock Cubical agda: A dependently typed programming language with
  univalence and higher inductive types.
\newblock {\em Proc. ACM Program. Lang.}, 3(ICFP), July 2019.

\bibitem[vO06]{vanoostenrelrec}
Jaap van Oosten.
\newblock {A General Form of Relative Recursion}.
\newblock {\em Notre Dame Journal of Formal Logic}, 47(3):311 -- 318, 2006.

\bibitem[vO08]{vanoosten}
Jaap van Oosten.
\newblock {\em Realizability: An Introduction to its Categorical Side}, volume
  152 of {\em Studies in Logic and the Foundations of Mathematics}.
\newblock Elsevier, North Holland, 2008.

\end{thebibliography}

\end{document}